\documentclass[11pt]{article}
\usepackage{amssymb}
\usepackage{amsfonts}
\usepackage{amsmath}
\usepackage{mathrsfs}
\usepackage{graphicx}
\usepackage{amsbsy}
\usepackage{theorem}
\usepackage{color}
\usepackage{hyperref}
\usepackage{tikz}
\usepackage[normalem]{ulem}

\newtheorem{theorem}{Theorem}[section]
\newtheorem{lemma}[theorem]{Lemma}

\newtheorem{definition}[theorem]{Definition}
\newtheorem{proposition}[theorem]{Proposition}
\newtheorem{remark}[theorem]{Remark}

\def\thetheorem{\thesection.\arabic{theorem}}
\def\thesection{\arabic{section}}

\def\theequation {\thesection.\arabic{equation}}
\def\beq{\begin{equation}\displaystyle}
\def\eeq{\end{equation}}
\def\bel{\begin{equation} \displaystyle \begin{array}{l} }
\def\eel{\end{array} \end{equation} }
\def\bell{\begin{equation} \displaystyle \begin{array}{ll}  }
\def\eell{\end{array} \end{equation} }

\def\bea{\begin{eqnarray}}
\def\eea{\end{eqnarray} }
\def\bean{\begin{eqnarray*}}
\def\eean{\end{eqnarray*} }
\newenvironment{proof}{\noindent{\bf Proof.~}}
{{\mbox{}\hfill {\small \fbox{}}\\}}
\def\qed{\mbox{}\hfill {\small \fbox{}}\\}
\catcode`@=11
\renewcommand\appendix{\bigskip {\noindent \Large \bf Appendix}
  \setcounter{section}{0}%
  \setcounter{subsection}{0}%
\setcounter{equation}{0}%
\setcounter{theorem}{0}%
\def\thetheorem{A.\arabic{theorem}}
\def\theequation {A.\arabic{equation}}}
\catcode`@=12

\def\NN{\mathbb{N}}

\def\RR{\mathbb{R}}

\def\ZZ{\mathbb{Z}}

\def\VV{\mathbb{V}}
\def\DD{\mathbb{D}}

\def\ds{\displaystyle}
\def\bs{\bigskip}

\def\eps{\varepsilon}
\def\bar#1{{\overline #1}}

\def\pa{\partial}

\def\calE{{\mathcal E}}

\def\calL{{\mathcal L}}
\def\calM{{\mathcal M}}
\def\calN{{\mathcal N}}

\def\calJ{{\mathcal J}}
\def\calS{{\mathcal S}}

\def\calV{{\mathcal V}}
\def\calH{{\mathcal R}}
\def\ds{{\displaystyle}}

\def\Mtilde{\widetilde{M}}
\def\barf{\bar{f}}
\newcommand {\ftilde} {\widetilde{f}}

\begin{document}

\title{Some examples of kinetic schemes whose diffusive limit is Il'in's exponential-fitting}

\author{L. Gosse\thanks{Istituto per le Applicazioni del Calcolo, via dei Taurini, 19, 00185 Rome, Italy, Email: \texttt{l.gosse@ba.iac.cnr.it} },
N. Vauchelet\thanks{Universit\'e Paris 13, Sorbonne Paris Cit\'e, CNRS UMR 7539, 
Laboratoire Analyse G\'eom\'etrie et Applications,
93430 Villetaneuse, France.  Email: \texttt{vauchelet@math.univ-paris13.fr}}}

\maketitle

\begin{abstract}
This paper is concerned with diffusive approximations of peculiar numerical schemes for several linear (or weakly nonlinear) kinetic models which are motivated by wide-range applications, including radiative transfer or neutron transport, run-and-tumble models of chemotaxis dynamics, and Vlasov-Fokker-Planck plasma modeling. The well-balanced method applied to such kinetic equations leads to time-marching schemes involving a ``scattering $S$-matrix'', itself derived from a normal modes decomposition of the stationary solution. One common feature these models share is the type of diffusive approximation: their macroscopic densities solve drift-diffusion systems, for which a distinguished numerical scheme is Il'in/Scharfetter-Gummel's ``exponential fitting'' discretization. We prove that the well-balanced schemes relax, within a parabolic rescaling, towards the Il'in exponential-fitting discretization by means of an appropriate decomposition of the $S$-matrix.
This is the so-called asymptotic preserving (or uniformly accurate) property.
\end{abstract}

\bs

\bs

\section{Introduction and contextualization}

\subsection{General setup}

Drift-diffusion equations, like (\ref{DDIntro}), arise naturally as diffusive approximations of numerous kinetic equations when time $t$ and space $x$ variables are conveniently rescaled. Such parabolic equations, in a context of semiconductor modeling, sparkled the development of so--called ``uniformly accurate'' (nowadays rephrased ``asymptotic-preserving'' (AP)) numerical methods: see e.g.  \cite{dorf,allen,gartland,IL,roos1,SG}. A thorough survey of such algorithms is presented in the book \cite{roos-book}. 
The ``exponential-fitting'' Il'in/Scharfetter-Gummel algorithm realizes one of the first well-balanced (WB) schemes for a parabolic equation in divergence form, like drift-diffusion equations, and is uniformly accurate (or AP), too, in the vanishing viscosity limit (contrary to the more standard Crank-Nicolson method).
Recently well-balanced numerical methods have been proposed in \cite[Part II]{book} to discretize kinetic equations.
The formulation of these schemes involves so-called scattering $S$-matrices built on exponential ``Knudsen layers''.
Accordingly, it is quite natural to wonder how this approach may be adapted to build well-balanced numerical methods which
\begin{enumerate}
\item are ``uniformly accurate'' (or AP) within a diffusive rescaling of variables~?
\item lead asymptotically to an ``exponential-fitting'' discretization ?
\end{enumerate}
In this paper, we intend to give a positive answer to both these questions for three examples of kinetic equations for which an explicit form of the $S$-matrix is known. Let $f(t,x,v)$ be a distribution function, depending on time $t>0$, position $x\in\RR$, and velocity $v\in V$: we shall consider,
\begin{itemize}
\item a first kinetic model, in parabolic scaling, which reads
\begin{equation}\label{eq1:kinIntro}
\eps \pa_t f^\eps + v\pa_x f^\eps = \frac{1}{\eps}\left(\int_{-1}^1 T_\eps(t,x,v')f^\eps(t,x,v')\,\frac{dv'}{2} - T_\eps(t,x,v)f^\eps(t,x,v)\right).
\end{equation}
When $T_\eps \equiv 1$, the well-known conservative radiative transfer equation is recovered,  
which, as $\eps\to 0$, approaches the heat equation.
When modeling chemotactic motions of bacteria, equation \eqref{eq1:kinIntro} is the so-called Othmer-Alt model \cite{othill}. The tumbling rate $T_\eps$ describes the response to variations of chemical concentration along a bacteria's path.
When the parameter $\eps\to 0$, it is now well-established that the macroscopic density $\rho:=\int_{-1}^1 f(v)\,dv$ solves the Keller-Segel system \cite{cmps}.
\item a related model, the Vlasov-Fokker-Planck system, for which the integral collision term is reduced to a diffusion operator. It reads, in parabolic scaling,
\begin{equation}\label{eq2:kinIntro}
\eps \pa_tf^\eps+v\pa_xf^\eps + E \cdot \pa_v f^\eps =
\frac{1}{\eps} \pa_v\left (vf^\eps + \kappa \pa_v f^\eps\right).
\end{equation}
It converges, as $\eps\to 0$ towards the drift-diffusion equation,  \cite{NPS,PS,WO1}.
\end{itemize}

\subsection{Scope and plan of the paper}

An object lying at the center of our matters is the so--called ``exponential-fit'' (Il'in \cite{IL}, Scharfetter-Gummel \cite{SG}, or Chang-Cooper \cite{chang}) numerical scheme for 1D drift-diffusion equations, that we briefly recall now. Consider,
\begin{equation}\label{DDIntro}
\pa_t \rho - \pa_x(\DD \pa_x \rho - E \rho)=0. \qquad 0 \leq \DD,\quad E \in \RR.
\end{equation}
This equation is discretized by means of a conservative numerical flux:
\begin{equation}\label{SGIntro}
\frac{\rho_j^{n+1}-\rho^n_j}{\Delta t}- \frac{F^n_{j+\frac 1 2}-F^n_{j-\frac 1 2}}{\Delta x}=0,
\end{equation}
where $\rho_j^n$ is an approximation of $\rho(t^n,x_j)$.
The flux $F^n_{j-\frac 1 2}$ is an approximation of $J:=\DD \pa_x \rho - E \rho$ at each interface of the grid $x_{j-\frac 1 2}$, which is derived by taking advantage of stationary solutions.
$$
J=\DD \pa_x \bar \rho(x) - E\,\bar \rho(x), \qquad \bar \rho(0)=\rho^n_{j-1}, \quad \bar \rho(\Delta x)=\rho^n_j.
$$
Then, elementary calculations lead to:
\begin{equation}\label{FluxSGIntro}
F^n_{j-\frac 1 2}:=J=E\,\frac{\rho^n_{j-1}-\exp(-E\,\Delta x / \DD)\,\rho^n_{j}}{1-\exp(-E\,\Delta x / \DD)}
\end{equation}
Scheme \eqref{SGIntro}--\eqref{FluxSGIntro} constitutes the classical Il'in/Sharfetter Gummel 
scheme \cite{gartland,IL,roos1,SG}. Our main result may be formulated as:
\begin{theorem}\label{main-thm}
Let (\ref{eqgen:scheme})--(\ref{eqgen:scat}) be a numerical scheme relying on a $S$-matrix for any of the linear $1+1$ kinetic equations (\ref{eq1:kinIntro}) and (\ref{eq2:kinIntro}) in parabolic scaling. Then, for $0 \leq \eps \ll 1$, and uniformly in $\Delta x>0$,
\begin{itemize}
\item each one, among the three considered $S$-matrices acting in (\ref{eqgen:scheme})--(\ref{eqgen:scat}), admits a decomposition of the type (\ref{decompS}); such a decomposition yields a Well-Balanced/Asymptotic-Preserving IMEX scheme like (\ref{schemeAPWB}); 
\item when $\eps\to 0$ in the numerical scheme, the corresponding macroscopic densities satisfy the Il'In/Sharfetter-Gummel scheme \eqref{SGIntro}--\eqref{FluxSGIntro}.
 \end{itemize}
\end{theorem}
Obviously, such a general statement contains several former ones, among which the two-stream Goldstein-Taylor model relaxing to the heat equation, \cite{GT}, or the so--called ``Cattaneo model of chemotaxis'', \cite{jmaa0}; some of these results were surveyed in \cite[Part II]{book}.
Yet, as a guideline for more involved calculations, we first explain in \S\ref{sec:2} how the limiting process works for a simple two-stream approximation of (\ref{eq1:kinIntro}), the so--called ``Greenberg-Alt'' model of chemotaxis \cite{ga}.
We mention that for this simple two-velocity model, a numerical scheme formulated in terms of (\ref{eqgen:scheme}) with $2 \times 2$ $S$-matrices is provided in both \cite[page 158]{book} and \cite[Lemma 4.1]{sisc} (for the purpose of hydrodynamic limits, though). Another type of closely related ``diffusive limit'' involving a $2 \times 2$ $S$-matrix was studied in \cite{BIT}.

This elementary calculation carried out on the two-stream ``Greenberg-Alt'' model (\ref{eq:2vit}) reveals why it is rather natural to expect that a well-balanced algorithm (\ref{eq:scheme}) (based on stationary solutions) may relax, within a parabolic scaling, toward the exponential-fit scheme (\ref{SG1-KS})--(\ref{SG2-KS}) for the corresponding asymptotic Keller-Segel model. However, as our Theorem \ref{main-thm} covers also conti\-nuous-velocity models discretized with general quadrature rules, we present in \S\ref{sec:3} our strategy of proof: in particular, the general scheme involving a scattering matrix is presented in (\ref{eqgen:scheme}) and the importance of the decomposition of the scattering matrix (\ref{decompS}) is emphasized. In \S4, such a strategy is applied to the simplest case of continuous equation, namely the ``grey radiative transfer'' model (\ref{eq:tr}). For this system, it is shown in Theorem \ref{th:RTE} that our numerical scheme  relaxes to the finite-difference discretization of the heat equation (\ref{rho-FD}). In \S\ref{sec:5}, the case of the Othmer-Alt \cite{othill} model of chemotaxis dynamics (\ref{eq:kinchemo})--(\ref{eq:Teps}) is handled in a similar manner (at the price of more intricate computations, though), yielding asymptotically the scheme (\ref{SGchemo}), this is Theorem \ref{th:chemo}. At last, in \S\ref{sec:6}, the case of a Vlasov-Fokker-Planck model (\ref{VFP}) is studied, and its asymptotic convergence towards (\ref{SGvfp}) is studied. 

An essential difference between Vlasov-Fokker-Planck kinetic models and the ones involving an integral collision operator (\ref{eq:tr}), (\ref{eq:kinchemo}) is that, being exponential polynomials, stationary solutions of  (\ref{VFP}) may not constitute Chebyshev $T$-systems on $v \in (0,+\infty)$; definition of Chebyshev $T$-systems is recalled in \S\ref{sec:3} below.
This drawback has to be compensated by supplementary assumptions on the set of discrete velocities, like (\ref{hypV1}) and (\ref{hypV2}).
Accordingly, general properties of eigenfunctions for each of the stationary kinetic models are stated in Appendix, along with a new result on exponential monomials, see \cite{Krat}.
%
\begin{remark}[Notations]
When $u\in \RR^N$ and $v\in \RR^M$, the matrix $u\otimes v$ is an element of
$\calM_{N\times M}$ whose coefficients are $(u_k v_\ell)$.
We will also commonly use the abuse of notations 
$$
\frac{1}{1+u\otimes v} \in \calM_{N\times M}, \mbox{ with coefficients } \left(\frac{1}{1+u_k v_\ell}\right)_{k,\ell}.
$$
\end{remark}
The present work somehow completes the former ones \cite{AML,sisc,jmaa} where hydrodynamic limits, involving finite-time concentrations, were considered; hereafter, diffusive limits yielding smooth solutions are studied, and general conclusions are identical: the schemes proposed in \cite[Part II]{book}, involving $S$-matrices based on stationary solutions, yield more accurate discretizations of the asymptotic regime. Namely, the upwind scheme for hydrodynamic limit (instead of Lax-Friedrichs, \cite{JVnum}), and exponential-fitting in the diffusive one.

\section{A general strategy for proving Main Theorem \ref{main-thm}}\label{sec:3}

\subsection{Haar property, Chebyshev $T$-systems and Markov systems}

We first recall basic notions from standard (one-dimensional) approximation theory, following mostly \cite[Chapter 3]{cheney}.
\begin{definition}\label{haar-def}
Let $n \in \NN$ and $F_n=(f_1, f_2, ..., f_n)$ be a family of functions, continuous on an interval $I \subset \RR$: it is endowed with the \underline{Haar property} if, for any strictly increasing family $X=(x_1, x_2, ..., x_n) \in I^n$, the  family  of $n$ vectors $(f_1(X), f_2(X), ..., f_n(X))$ is linearly independent. Equivalently, the determinant never vanishes:
$\forall\,(x_1, x_2, \ldots , x_n) \in I^n,\ x_1<x_2<\ldots<x_n$,
\begin{equation}\label{haar-det}
\left|\begin{array}{cccc}
f_1(x_1) & f_2(x_1)& \cdots & f_n(x_1) \\
f_1(x_2) & f_2(x_2)& \cdots & f_n(x_2) \\
\vdots & \vdots & \ddots & \vdots \\
f_1(x_n) & f_2(x_n)& \cdots & f_n(x_n) \\
\end{array}\right|
\not = 0.
\end{equation}
Such a family $F_n$ constitutes a \underline{Chebyshev $T$-system} on the interval $I$.
\end{definition}
The simplest example of $T$-system on $I=\RR$ is the monomials family, for which the determinant (\ref{haar-det}) is the well-known {\it Vandermonde} determinant.
\begin{definition}\label{markov-def}
Let $F=(f_1, f_2, ...)$ be an infinite sequence of  functions, continuous on an interval $I \subset \RR$: it is said to be a  \underline{Markov system} if, for any $n \in \NN$, the extracted finite family $F_n \subset F$ is a Chebyshev $T$-system.
\end{definition}
A standard, yet important result is the following:
\begin{proposition}\label{prop1:Haar}
Let $n \in \NN$ and $F_n=(f_1, f_2, ..., f_n)$ be as in Definition \ref{haar-def}: it is a $T$-system on $I$ if and only if any (real) linear combination,
\begin{equation}\label{T-roots}
\forall (a_1, a_2, \cdots, a_n) \in \RR^n, \qquad I \ni x \mapsto \sum_{i=1}^n a_i\, f_i(x),
\end{equation}
admits at most $n-1$ real roots on $I$.
\end{proposition}
\begin{remark}\label{rem-hadamard}
Determinants of the type (\ref{haar-det}) are called ``alternant determinants'', see \cite[Chapter 4]{vein}. Moreover, the Haar property is closely related to ``total positivity'' of 
matrices, see {\it e.g.} \cite{bojanov}. Being the ``Hadamard product'' the component-wise product of two $n \times n$ matrices,
$$
(A \circ B)_{1 \leq i,j, n}:= A_{i,j}\, B_{i,j},
$$
Garloff and Wagner, in \cite[page 100]{bojanov}, explain that the Haar property is not generally preserved by multiplying elements of two $T$-systems with each other. A first exception is given by two generalized Vandermonde matrices $G_n=(x_i^{\alpha_j})_{1 \leq i,j, n}$ sharing either the set of points $X$ or the exponents $\alpha_i$'s. A second one is given by non-negative exponential monomials: see Theorem \ref{th:alternant}.
\end{remark}

\subsection{General strategy for building AP and WB schemes}
\label{sec:strategy}

Let us first give a Discrete-Ordinates (DO) setup.
A spatial domain is gridded with nodes $x_j=x_0+ j \Delta x$, $j \in \ZZ$ 
and a velocity domain, $V$, symmetric with respect to $0$, by
$v_k$, $k=-K,\ldots,-1,1,\ldots,K$, $0<v_1<v_2<\ldots<v_K$, and $v_{-k}=-v_k$.
We denote
$$
\calV=(v_1,\ldots,v_{K})^\top\in \RR^K, \qquad \VV:=\mbox{diag}(v_1,\ldots,v_{K},v_1,\ldots,v_{K})\in \calM_{2K}(\RR).
$$
Corresponding weights $(\omega_k)_{k=1,\ldots,K}$ may be given by a Gauss quadrature,
so 
$$
\int_V \phi(v)dv \mbox{ is approximated by } \sum_{k=1}^{K} \omega_k(\phi(v_k)+\phi(-v_k)).
$$
The time step will be denoted $\Delta t >0$.

We consider a kinetic system in parabolic scaling
$$
\eps \pa_t f + v \pa_x f = \frac{1}{\eps} \calL(f),\qquad t,x,v, \in \RR_*^+ \times \RR \times V,
$$
where $\calL$ is a linear operator, depending on the nature of the given problem.
We assume that when $\eps\to 0$ the macroscopic density, defined by $\rho:=\int_V f(v)\,dv$,
converges to a solution of a drift-diffusion kind of equation \eqref{DDIntro}.

Let a discretization of the distribution function 
at time $t^n$, be $(f^n_j(\pm v_k))_{j,k}$.
Our general strategy for building AP and WB schemes follows these steps.
\begin{itemize}
\item {\bf 1st step. Determination of the $S$-matrix.}
In order to build a well-balanced scheme, i.e. which preserves equilibria,
it is important to be able to compute stationary solutions.
Let us consider the following stationary problem
with incoming boundary conditions on $(0,\Delta x)$ for each $j=1,\ldots,N_x$,
\begin{equation}\label{eqstat}
\eps v \pa_x \barf = \calL_{j-\frac 12}(\barf) , \quad 
\barf(0,\calV) = f_{j-1}(\calV), \quad \barf(\Delta x,-\calV) = f_j(-\calV).
\end{equation}
In this system $\calL_{j-\frac 12}$ is a discretization of $\calL$ on $(x_{j-1},x_j)$
such that $\calL_{j-\frac 12}$ is a linear operator. 
Notice that it is enough to solve the problem 
with $\eps=1$ thanks to the change of variable 
$x\rightarrow x/\eps$.
The unknwon of the problem is the function $\barf$ and we want 
to determine the outgoing flux $\begin{pmatrix} \barf(\Delta x,\calV) \\ \barf(0,-\calV)  \end{pmatrix}$.
Since \eqref{eqstat} is linear, 
the computation of the outgoing flux involves a so-called scattering matrix
$\calS_{j-1/2}^\eps$ defined by
\begin{equation}\label{eqgen:scat}
\begin{pmatrix} \barf_{j-\frac 12}(\calV) \\ \barf_{j-\frac 12}(-\calV) \end{pmatrix} :=
\begin{pmatrix} \barf(\Delta x,\calV) \\ \barf(0,-\calV)  \end{pmatrix} =
\calS_{j-\frac 12}^\eps \begin{pmatrix} f_{j-1}(\calV) \\ f_j(-\calV) \end{pmatrix}.
\end{equation}
Several $S$-matrices for various kinetic models 
are provided in \cite[Part II]{book}, including the approach 
based on Case's elementary solutions \cite{ac}.
\item {\bf 2nd step. Well-balanced scheme.}
Once the scattering matrix is known, one may define the well-balanced scheme as (see \cite{mms})
\begin{align}\label{eqgen:scheme}
\begin{pmatrix} f_j^{n+1}(\calV) \\ f_{j-1}^{n+1}(-\calV) \end{pmatrix} &=
\begin{pmatrix} f_j^n(\calV) \\ f_{j-1}^n(-\calV) \end{pmatrix} 
- \frac{\Delta t}{\eps \Delta x}\VV
\begin{pmatrix} f_j(\calV) - \barf_{j-\frac 12}(\calV)  \\ 
f_{j-1}(-\calV)-\barf_{j-\frac 12}(-\calV) \end{pmatrix} 
\end{align}
It verifies the well-balanced property, {\it  i.e.} stationary states are preserved. 
\item {\bf 3rd step. Asymptotic preserving scheme.}
Obviously to have an uniformly accurate scheme and to be able to pass to the limit as $\eps\to 0$ into the scheme
\eqref{eqgen:scheme}, we need to treat implicitly terms in
$\frac{1}{\eps}$. To do so, a crucial step is the decomposition
\begin{equation}\label{decompS}\boxed{
\calS_{j-\frac 12}^\eps = \calS_{j-\frac 12}^0 + \eps \calS_{j-\frac 12}^{1,\eps}.}
\end{equation}
Finally, the scheme \eqref{eqgen:scheme}--\eqref{eqgen:scat} becomes
\begin{align}\label{schemeAPWB}
\begin{pmatrix} f_j^{n+1}(\calV) \\ f_{j-1}^{n+1}(-\calV) \end{pmatrix} 
&+ \frac{\Delta t}{\eps \Delta x}\VV\left[
\begin{pmatrix} f_j^{n+1}(\calV) \\ f_{j-1}^{n+1}(-\calV) \end{pmatrix}
- \calS_{j-\frac 12}^0 \begin{pmatrix} f_{j-1}^{n+1}(\calV) \\ f_{j}^{n+1}(-\calV) \end{pmatrix}
\right] \\
&= \begin{pmatrix} f_j^n(\calV) \\ f_{j-1}^n(-\calV) \end{pmatrix} 
 +\frac{\Delta t}{\Delta x} \VV \calS_{j-\frac 12}^{1,\eps} 
\begin{pmatrix} f_{j-1}^{n}(\calV) \\ f_{j}^{n}(-\calV) \end{pmatrix}. \nonumber
\end{align}
Recall that an approximation of macroscopic density is recovered thanks to the quadrature by
\begin{equation}\label{eq:rhof}
\forall j,n, \qquad \rho_j^n:=\sum_{k=1}^K \omega_k(f_j^n(v_k)+f_j^n(-v_k)).
\end{equation}
\end{itemize}

\subsection{Strategy for the proof of the main result}\label{sec:idea}

The aim of this paper is to prove that, at least for the kinetic systems \eqref{eq1:kinIntro}
and \eqref{eq2:kinIntro} for
which scattering matrices are well established, the limit as $\eps\to 0$ of scheme \eqref{schemeAPWB}
leads to the Il'In/Sharfetter Gummel scheme \eqref{SGIntro}--\eqref{FluxSGIntro} 
for the macroscopic density defined in \eqref{eq:rhof}.
A first ingredient in the proof will be to establish that actually the leading order term
in the decomposition \eqref{decompS} writes with $K\times K$ block matrices
$$
\calS^0_{j-\frac 12} = \begin{pmatrix} \mathbf{0}_K & \calS^{0}_{1,j-\frac 12} \\ \calS^{0}_{1,j-\frac 12} & \mathbf{0}_K
\end{pmatrix} \quad \mbox{ and } \quad
\calS^{1,\eps}_{j-\frac 12} = \begin{pmatrix} \calS^{1,\eps}_{1,j-\frac 12} & \calS^{1,\eps}_{2,j-\frac 12} \\
  \calS^{,\eps}_{3,j-\frac 12} & \calS^{1,\eps}_{4,j-\frac 12} \end{pmatrix}.
$$
As a consequence, scheme \eqref{schemeAPWB} can be recast in IMEX 
(IMplicit-EXplicit \cite{boparu,paru}) form,
\begin{equation}\label{IMEX}
\calH_\eps \begin{pmatrix} f_j^{n+1}(\calV) \\ f_j^{n+1}(-\calV) \end{pmatrix}
= \begin{pmatrix}\eps  f_j^{n}(\calV) \\\eps  f_j^{n}(-\calV) \end{pmatrix}
+ \frac{\eps \Delta t}{\Delta x} \VV 
\begin{pmatrix} S_{1,j-\frac 12}^{1,\eps} f_{j-1}^n(\calV)+S_{2,j-\frac 12}^{1,\eps} f_j^n(-\calV) \\
 S_{3,j+\frac 12}^{1,\eps} f_{j}^n(\calV)+S_{4,j+\frac 12}^{1,\eps} f_{j+1}^n(-\calV) \end{pmatrix},
\end{equation}
for a matrix $\calH_\eps$ (which details depend on each kinetic model) given by,
$$
\calH_\eps = \eps {\mathbf I}_{\RR^{2K}} + \frac{\Delta t}{\Delta x} \VV
\begin{pmatrix} \mathbf{I}_K & -\calS^{0}_{1,j-\frac 12} \\ -\calS^{0}_{1,j+\frac 12} & \mathbf{I}_K
\end{pmatrix}
\underset{\eps\to 0}{\longrightarrow} \calH_0:=
\frac{\Delta t}{\Delta x} \VV
\begin{pmatrix} \mathbf{I}_K & -\calS^{0}_{1,j-\frac 12} \\ -\calS^{0}_{1,j+\frac 12} & \mathbf{I}_K
\end{pmatrix}.
$$ 
To ensure global solvability of such a scheme in the diffusive limit, we shall argue according to the analysis performed for continuous equations, see for instance \cite[Chap. 5]{ansgar}, namely by proving that:
\begin{itemize}
\item Ker($\calH_0$) is a vectorial line, closely related to Maxwellian distributions;
\item its range is an hyperplane which contains all kinetic distributions with vanishing macroscopic densities ({\it i.e.} null moments of order zero).
\end{itemize}
These conditions are discrete analogues of the Fredholm alternative holding for continuous limits. 
Yet, we write that kinetic densities have a Hilbert expansion, $f=f^0+\eps f^1 + \ldots$.
Injecting into \eqref{IMEX}, we first deduce by identifying the term at order 0 in $\eps$ 
that $f^0\in \mbox{Ker}(\calH_0)$. Identifying the terms
at order 1 in $\eps$, we get
\begin{align}
\calH_0 \begin{pmatrix} \{f^1\}_j^{n+1}(\calV) \\[1mm] \{f^1\}_j^{n+1}(\calV) \end{pmatrix} = 
&\begin{pmatrix} (\{f^0\}_j^{n}-\{f^0\}_j^{n+1})(\calV) \\[1mm] (\{f^0\}_j^{n}-\{f^0\}_j^{n+1})(-\calV) \end{pmatrix}
\label{eqR0Intro}  \\[1mm]   
& + \frac{\Delta t}{\Delta x} \VV 
\begin{pmatrix} S_{1,j-\frac 12}^{1,0} \{f^0\}_{j-1}^n(\calV)+S_{2,j-\frac 12}^{1,0} \{f^0\}_j^n(-\calV) \\[1mm]
 S_{3,j+\frac 12}^{1,0} \{f^0\}_{j}^n(\calV)+S_{4,j+\frac 12}^{1,0} \{f^0\}_{j+1}^n(-\calV) \end{pmatrix}. \nonumber
\end{align}
This equation admits a solution iff the right hand side belongs to the range of $\calH_0$.
Taking moments of order zero allows us to deduce the asymptotic discretization governing the numerical macroscopic density.

To summarize, the main issues are: 
\begin{itemize}
\item to compute the scattering matrix $\calS$ and determine its decomposition \eqref{decompS};
\item to study the kernel and the range of $\calH_0$;
\item to establish that taking the moment of order zero of \eqref{eqR0Intro} leads to Il'In/Sharfetter Gummel scheme for macroscopic density.
\end{itemize}



Finally, for future use, we recall the following result:
\begin{lemma}[Lemma 3.1 in \cite{CGT}, Prop. 1 in \cite{mms}]\label{lem:CGT}
Let $\Gamma$ be the following diagonal matrix, 
$$\Gamma = \begin{pmatrix} \mbox{diag}(\omega_k v_k)_{k=1,\ldots,K} & 0_K \\ 0_K & \mbox{diag}(\omega_k v_k)_{k=1,\ldots,K} \end{pmatrix},
$$
then, being given a nonnegative initial data $f^0_j(\pm\calV)_j$, the scheme (\ref{eqgen:scheme}) 
preserves both non-negativity and the (discrete) $L^1$ norm of $f_j^n(\pm \calV)$ as soon as
\begin{align}
&\max(\VV) \Delta t \leq \eps \Delta x,  \quad \mbox{ CFL condition,}  \label{CFL}  \\
&\forall\, j,\quad \Gamma \calS_{j-1/2}^n \Gamma^{-1} \ \mbox{ is left-stochastic (each column summing to 1)}. \label{eq:massc}
\end{align}
Moreover, it preserves its $L^\infty$ norm as well if $\calS_{j-1/2}^n$ is right-stochastic (each row summing to 1).
\end{lemma}

\section{The two-stream Greenberg-Alt kinetic model}\label{sec:2}

For the sake of simplicity,
we first start our exposition by a very simple model consisting in
a two-velocity kinetic model.

\subsection{Diffusive limit of the continuous system}

The simplest two-velocity kinetic model
describing the motion of bacteria by chemotaxis was given in \cite{ga}. 
Let $f^+$ ($f^-$) be the distribution function of right-moving (left-moving) 
bacteria, the following system describes their motion governed by a
{\it run and tumble} process (see also e.g. \cite{cmps,ds})
\begin{equation}\label{eq:2vit}
\eps \pa_t f^\pm \pm \pa_x f^\pm = \pm \frac{1}{2\eps}
\Big( \big( 1+\eps \phi(\pa_x S)) f^- - (1-\eps\phi(\pa_xS)) f^+\Big).
\end{equation}
The quantity $S$ is the chemoattractant concentration, and solves 
\begin{equation}\label{eq:S}
-\pa_{xx} S + S = \rho.
\end{equation}
Macroscopic quantities being $\rho=f^++f^-$ (density) and $J=\frac{1}{\eps}(f^+-f^-)$ (current),
adding and subtracting former equations yields, 
$$
\pa_t \rho + \pa_x J = 0, \qquad 
\eps^2 \pa_t J + \pa_x \rho = \phi(\pa_x S) \rho - J.
$$
Letting formally $\eps\to 0$ in the second equation,
$$
J = \phi(\pa_xS) \rho - \pa_x \rho,
$$
and yields the well-known Keller-Segel,
system when coupled with (\ref{eq:S}),
\begin{equation}\label{eq:KS}
\pa_t \rho - \pa_{xx} \rho + \pa_x \big(\phi(\pa_x S)\rho\big) = 0.
\end{equation}

\subsection{Exponential-fit scheme for Keller-Segel equation}

We first recall the Il'in/Sharfetter Gummel scheme for system \eqref{eq:KS},
which writes under the form \eqref{DDIntro} with $\DD=1$ and $E=\phi(\pa_xS)$.
Assuming that approximations $(\rho_j^n)_j$ of $\rho(t^n,x_j)$ and 
$(S_j^n)_j$ of $S(t^n,x_j)$ are available, 
we denote 
$$
\forall j, \qquad \phi_{j-1/2}^n = \phi\Big(\frac{S_{j}^n-S_{j-1}^n}{\Delta x}\Big).
$$
Then, the Il'In/Sharfetter Gummel scheme \eqref{SGIntro}--\eqref{FluxSGIntro} reads,
in this framework,
\begin{equation}\label{SG1-KS}
\bar{\calJ}_{j-1/2}^n = \frac{\phi_{j-1/2}^n}{1-e^{-\phi_{j-1/2}^n\Delta x}}
\big(\rho_{j-1}^n - e^{-\phi_{j-1/2}^n \Delta x} \rho_j^n\big),
\end{equation}
which gives (constant) currents. Then densities are updated by,
\begin{equation}\label{SG2-KS}
\rho_j^{n+1} = \rho_j^n + \frac{\Delta t}{\Delta x}
\big(\bar{\calJ}_{j-1/2}^n-\bar{\calJ}_{j+1/2}^n\big).
\end{equation}

\subsection{Asymptotic preserving and well-balanced scheme}

We follow the strategy proposed in \S\ref{sec:strategy}.
Let us assume that $(f_j^{n,+},f_j^{n,-})_{j}$ are known at time $t^n$ along 
with an approximation $S_j^n$ of $S(t^n,x_j)$, giving 
$\phi_{j-1/2}^n = \phi(\frac{S_{j}^n-S_{j-1}^n}{\Delta x})$.

\begin{itemize}
\item {\bf 1st step. Scattering matrix.}
It is computed by solving the stationary system in $(x_{j-1},x_j)$ 
with incoming boundary conditions,
\begin{equation}\label{eq:stat}
\left\{\begin{array}{l}
\pa_x \bar{f}^\pm = \ds \frac{1}{2\eps} \big((1+\eps\phi_{j-1/2}^n) \bar{f}^-
         -(1-\eps\phi_{j-1/2}^n) \bar{f}^+\big)   \\[2mm]
\bar{f}^+(x_{j-1}) = f^{+}_{j-1} \quad ; \qquad \bar{f}^-(x_j) = f_j^{-}.
\end{array}\right.
\end{equation}
The unknown for this problem are interface values $\bar{f}_{j-1/2}^+:=\bar{f}^+(x_j)$ and  
$\bar{f}_{j-1/2}^-:=\bar{f}^-(x_{j-1})$.
System \eqref{eq:stat} may be solved exactly by, first, subtracting both equations,
$$
\pa_x(\bar{f}^+ - \bar{f}^-) = 0, \qquad \bar{J}:=\frac{1}{\eps}(\bar{f}^+-\bar{f}^-)
\mbox{ is constant in } (x_{j-1},x_j).
$$
and then, adding them, so that by denoting $\bar{\rho}=\bar{f}^++\bar{f}^-$,
$$
\pa_x \bar{\rho} = \phi_{j-1/2}^n \bar{\rho} - \bar{J}, \qquad 
e^{-\phi_{j-1/2}^n \Delta x} \bar{\rho}_{j} - \bar{\rho}_{j-1} = \bar{J}
\frac{e^{-\phi\Delta x}-1}{\phi}.
$$
Thus the system to be solved is 
\begin{align*}
f^+_{j-1}-\bar{f}_{j-1/2}^- &= \bar{f}_{j-1/2}^+ - f_j^-  \\
e^{-\phi_{j-1/2}^n\Delta x}(f_j^-+\bar{f}_{j-1/2}^+)-(f_{j-1}^+ + \bar{f}_{j-1/2}^-)
&= \frac{e^{-\phi_{j-1/2}^n \Delta x}-1}{\eps \phi_{j-1/2}^n} 
(f_{j-1}^+ - \bar{f}_{j-1/2}^-).
\end{align*}
After easy computations,
\begin{align*}
\bar{f}_{j-1/2}^- = &f_{j-1}^+ + \frac{2 \eps \phi_{j-1/2}^n\big(f^+_{j-1} - e^{-\phi_{j-1/2}^n \Delta x} f_j^- \big)}{e^{-\phi_{j-1/2}^n\Delta x}-1
-\eps \phi_{j-1/2}^n(1+e^{- \phi_{j-1/2}^n \Delta x})} ,  \\
\bar{f}_{j-1/2}^+ = &f_{j}^- - \frac{2 \eps \phi_{j-1/2}^n\big(f^+_{j-1} - e^{-\phi_{j-1/2}^n \Delta x} f_j^- \big)}{e^{-\phi_{j-1/2}^n\Delta x}-1
-\eps \phi_{j-1/2}^n(1+e^{- \phi_{j-1/2}^n \Delta x})} .
\end{align*}
Following \cite[pp.157--58]{book} (or \cite{AML}), one may rewrites the latter system
by defining a $S$-matrix,
$$
\begin{pmatrix}
\bar{f}_{j-1/2}^+ \\ \bar{f}_{j-1/2}^-
\end{pmatrix}
=\calS^n_{j-1/2} \begin{pmatrix}
{f}_{j-1}^+ \\ {f}_{j}^-
\end{pmatrix},
$$
which (with shorthand notation $\calE= e^{-\phi_{j-1/2}^n \Delta x}$) reads
$$
\calS^n_{j-1/2}= \begin{pmatrix}
\frac{-2 \eps \phi_{j-1/2}^n}{\calE - 1-\eps \phi_{j-1/2}^n(1+\calE)} & 1+\frac{2 \eps \phi_{j-1/2}^n \, \calE}{\calE - 1-\eps \phi_{j-1/2}^n(1+\calE)} \\
 1+\frac{2 \eps \phi_{j-1/2}^n }{\calE - 1-\eps \phi_{j-1/2}^n(1+\calE)} & \frac{-2 \eps \phi_{j-1/2}^n \, \calE}{\calE - 1-\eps \phi_{j-1/2}^n(1+\calE)}
\end{pmatrix}.
$$
Notice that it is clearly left-stochastic.

\item{\bf 2nd step. WB scheme.}
Following \cite{GT}, we consider 
\begin{equation}\label{eq:scheme}
\left\{\begin{array}{l}
f_j^{+,n+1} = f_j^{+,n} - \frac{\Delta t}{\eps \Delta x} (f_j^{+,n+1} 
- \bar{f}_{j-1/2}^{+})   \\
f_{j-1}^{-,n+1} = f_{j-1}^{-,n} - \frac{\Delta t}{\eps \Delta x} (f_{j-1}^{-,n+1}
- \bar{f}_{j-1/2}^{-}).
\end{array}\right.
\end{equation}

\item{\bf 3rd step. AP scheme.}
For such a simple case, the decomposition (\ref{decompS}) is straightforward:
\begin{equation}\label{S-01}
\calS^n_{j-1/2}= \begin{pmatrix}
0 & 1 \\ 1 & 0 \end{pmatrix} + 2 \eps\cdot \frac{ \phi_{j-1/2}^n }{\calE - 1-\eps \phi_{j-1/2}^n(1+\calE)}
\begin{pmatrix}
-1 & \calE \\ 1 & -\calE \end{pmatrix}.
\end{equation}
By treating implicitly the first (stiff) term and plugging into the scheme (\ref{eqgen:scheme}),
\begin{equation}\label{eq:schem2}
\left\{\begin{array}{l}
f_j^{+,n+1} = f_j^{+,n} - \frac{\Delta t}{\eps \Delta x} (f_j^{+,n+1} 
- f_j^{-,n+1}) + \frac{\Delta t}{\Delta x} \bar{J}^n_{j-1/2}   \\[2mm]
f_{j}^{-,n+1} = f_{j}^{-,n} - \frac{\Delta t}{\eps \Delta x} (f_{j}^{-,n+1}
- f_{j}^{+,n+1}) - \frac{\Delta t}{\Delta x} \bar{J}^n_{j+1/2}  \\[2mm]
\bar{J}^n_{j-1/2} = \frac{ - 2 \phi_{j-1/2}^n\big(f^{n,+}_{j-1} - e^{-\phi_{j-1/2}^n \Delta x} f_j^{n,-} \big)}{e^{-\phi_{j-1/2}^n\Delta x}-1
-\eps \phi_{j-1/2}^n(1+e^{- \phi_{j-1/2}^n \Delta x})} . 
\end{array}\right.
\end{equation}
\end{itemize}
For this simple case, the limit $\eps\to 0$ may be performed easily.
Adding the first two equations in \eqref{eq:schem2}, we obtain,
\begin{equation}\label{eq:SG1}
\rho_j^n=f_j^{+,n}+f_j^{-,n}, \qquad 
\rho_j^{n+1} = \rho_j^{n} - \frac{\Delta t}{\Delta x}(\bar{J}^n_{j+1/2}-\bar{J}^n_{j-1/2}).
\end{equation}
Letting then $\eps\to 0$, we deduce easily from \eqref{eq:schem2} that 
$$
\bar{J}_{j-1/2}^n \to \frac{ - 2 \phi_{j-1/2}^n}{e^{-\phi_{j-1/2}^n\Delta x}-1} 
\big(f^{n,+}_{j-1} - e^{-\phi_{j-1/2}^n \Delta x} f_j^{n,-} \big).
$$
Yet, multiplying the first equations of \eqref{eq:schem2}
by $\eps$ and letting $\eps\to 0$, at the limit 
the relation $f_j^{+,n+1}=f_j^{-,n+1}$ is enforced.
Accordingly, the current rewrites
\begin{equation}\label{eq:SG2}
\bar{J}_{j-1/2}^n \to \frac{ - \phi_{j-1/2}^n}{e^{-\phi_{j-1/2}^n\Delta x}-1} 
\big(\rho^{n}_{j-1} - e^{-\phi_{j-1/2}^n \Delta x} \rho_j^{n} \big).
\end{equation}
Injecting  \eqref{eq:SG2} into \eqref{eq:SG1},
Il'in scheme for Keller-Segel system \eqref{eq:KS} is found.

\begin{remark}
All computations are explicit, so there is no real need of 
$\calH_\eps$ when passing to the limit. From \eqref{eq:schem2}, we get 
the matrix already met in \cite[\S3]{GT}, 
$$
\calH_\eps = \begin{pmatrix} \eps + \frac{\Delta t}{\Delta x} & -\frac{\Delta t}{\Delta x} \\ -\frac{\Delta t}{\Delta x} & \eps + \frac{\Delta t}{\Delta x}
\end{pmatrix}
\underset{\eps\to 0}{\longrightarrow} 
\calH_0 = \frac{\Delta t}{\Delta x}\begin{pmatrix} 1 & -1 \\ -1 & 1
\end{pmatrix}.
$$
Thus, Ker$(\calH_0)=$Vect$\{(1,1)^\top\}$ and $(u,v)^\top\in\mbox{Im}(\calH_0)$
iff $u+v=0$.
\end{remark}

\section{Heat equation as a diffusive limit of radiative transfer}\label{sec:4}

Let us consider the kinetic model \eqref{eq1:kinIntro} in the simple case $T_\eps\equiv 1$:
\begin{equation}\label{eq:tr}
\eps \pa_t f + v\pa_xf = \frac{1}{\eps}\left(\int_{-1}^1 f(t,x,v')\frac{dv'}{2} -f\right).
\end{equation}
The macroscopic density, at the limit $\eps\to 0$, satisfies the heat equation
\begin{equation}\label{hit_eq}
\pa_t \rho - \frac{1}{3} \pa_{xx} \rho = 0, \qquad \rho(t,x)=\int_{-1}^1 f(t,x,v)dv.
\end{equation}
Equation (\ref{eq:tr}) is usually referred to as to ``gray radiative transfer''; more specific models can be drawn by  replacing the uniform integral kernel $\frac 1 2$ by an even, nonnegative, function of the velocity variable, $0 \leq \mathcal{K}(v)=\mathcal{K}(-v)$,
$$
\eps \pa_t f + v\pa_xf = \frac{1}{\eps}\left({\mathcal{K}(v)} \, \int_{-1}^1 f(t,x,v'){dv'} -f\right), \qquad \int_{-1}^1 \mathcal{K}(v)=1.
$$

For this system, the velocity domain being $V=(-1,1)$,
we assume that the set $(\omega_k,v_k)$ introduced in \S\ref{sec:strategy} satisfies 
the mild restrictions,
\begin{equation}\label{eq:restrict}
\sum_{k=1}^K \omega_k = 1, \qquad  \sum_{k=1}^K \omega_k v_k^2 = \frac 1 3, \qquad v_{-k}=-v_k.
\end{equation}

\subsection{Solving the stationary problem}

Let us first investigate some useful properties on the stationary equation for general nonnegative kernel $T$:
\begin{proposition}\label{prop:lambda}
Let $\barf(x,v)=\exp(-\lambda x)\phi_\lambda(v)$ be a separated-variables solution of 
the stationary equation
\begin{equation}\label{eqstatdis}
v\, \pa_x \barf= \int_{-1}^1 T(v')\, \barf(x,v'){d\nu(v')} - T(v)\,\barf, \qquad v \in (-1,1),
\end{equation}
where $\nu$ is a probability measure on $(-1,1)$. 
Then, we have
\begin{equation}\label{eqphilam}
\phi_\lambda(v) = \frac{1}{T(v)-\lambda v}, \qquad 
\int_{-1}^1 \frac{T(v)}{T(v)-\lambda v} d\nu(v) = 1.
\end{equation}
Moreover, the following orthogonality relation holds
\begin{equation}\label{ortho}
\forall \lambda \not = \mu, \quad 
 \int_{-1}^1 v\,\phi_\lambda(v) \phi_\mu(v)\,  T(v)\, d\nu(v) = 0.
\end{equation}
In particular, for $\lambda=0$, 
$$
\phi_0(v)= \frac{1}{T(v)}, \qquad \forall \mu \not =0, \quad J_\mu=\int_{-1}^1 v\, \phi_\mu(v)\,d\nu(v) =0.
$$
\end{proposition}
The $\lambda$ are usually called the ``eigenvalues'' and the corresponding $\phi_\lambda$ are the
``Case's eigenfunctions''.
\begin{remark}
Under assumption \eqref{eq:restrict} on the discrete velocity set, we may choose
the measure $\nu=\frac 12 \sum_{k=-K}^K \omega_k \delta(v-v_k)$ in this proposition.
Then, relation \eqref{ortho} becomes
\begin{equation}\label{orthodiscret}
\forall\, \lambda\neq \mu, \qquad  \sum_{k=-N}^N \omega_k v_k \phi_\lambda(v_k) \phi_\mu(v_k) T(v_k) = 0.
\end{equation}
In particular, for $\lambda=0$, 
\begin{equation}\label{ortho0discret}
\forall\, \mu \neq 0, \qquad  \sum_{k=1}^N \omega_k v_k (\phi_\mu(v_k)-\phi_\mu(-v_k)) = 0.
\end{equation}
\end{remark}

\begin{proof}
Inserting the ansatz $\barf(x,v)=\exp(-\lambda x)\phi_\lambda(v)$ into \eqref{eqstatdis} implies
$$
(T(v)-\lambda v) \phi_\lambda(v) = \int_{-1}^1 T(v) \phi_\lambda(v) \,d\nu(v).
$$
By linearity, $\phi_\lambda$ is defined up to a constant, we may fix this constant by imposing
$$
\int_{-1}^1 T(v) \phi_\lambda(v) \,d\nu(v) = 1.
$$
It gives the relations \eqref{eqphilam}.
Then, for two eigenvalues $\lambda\neq \mu$, we have
$$
(T(v) - \lambda v)\phi_\lambda(v) = 1; \qquad (T(v) - \mu v) \phi_\mu(v) = 1.
$$
We multiply the first identity by $T(v)\phi_\mu(v)$, the second by $T(v)\phi_\lambda(v)$, 
and integrate over $d\nu(v)$, we obtain after subtracting the 
resulting identities
$$
(\lambda-\mu) \int_{-1}^1 v \phi_\lambda(v) \phi_\mu(v) T(v) \,d\nu(v) = 0.
$$
We deduce the orthogonality relation in \eqref{ortho}.
\qed
\end{proof}

\subsection{The scattering matrix and its decomposition (\ref{decompS})}
\label{sec:SRTE}

Following the strategy proposed in Section \ref{sec:strategy}, we first determine the 
scattering matrix and its expansion \eqref{decompS}.
The stationary problem with incoming boundary data reads, for an index $j$,
\begin{align}
&\eps v \pa_x \barf = \frac 12 \int_{-1}^1 \barf(v')dv' -\barf, \qquad \mbox{ on } (0,\Delta x),  
\label{eq:stat1} \\
& \barf(0,v) = f_{j-1}(v), \qquad \barf(\Delta x,-v) = f_j(-v),  
\label{eq:stat2} 
\end{align}
where $\big( f_{j-1}(|v|),  f_j(-|v|)\big)$ are ``incoming values'', and outgoing ones read:
$$
\barf_{j-1/2}(|v|)=\barf(\Delta x,|v|), \quad \barf_{j-1/2}(-|v|) = \barf(0,-|v|).
$$

In view of Proposition \ref{prop:lambda}, equation \eqref{eqphilam} in the particular case $T=1$, the discrete eigenelements are given by
$$
\frac 12 \sum_{k=1}^K \omega_k\left(\frac{1}{1-\lambda v_k}+\frac{1}{1+\lambda v_k}\right) = 1, \qquad
\phi_\lambda(v) = \frac{1}{1-\lambda v}.
$$
Clearly, we have that if $\lambda$ is an eigenvalue, then $-\lambda$ is also an eigenvalue and $\phi_{-\lambda}(v)=\phi_\lambda(-v)$.
We observe also that the eigenvalue $\lambda=0$ is double; the eigenvectors are $1$ and $x-\eps v$. (Indeed we verify easily that such functions solve \eqref{eq:stat1}). 
Therefore, a quite general stationary solution is obtained by truncating to the first $2K$ eigenmodes, (see {\it e.g.} \cite{ac,cerc,book})
$$
\barf(x,v) = a_0 + b_0(x - \eps v) + \sum_{\ell=1}^{K-1}  \left(\frac{a_\ell\, e^{-\lambda_\ell\, x/ \eps }}{1-\lambda_\ell\, v} + \frac{b_\ell\, e^{\lambda_\ell (x-\Delta x)/ \eps }}{1+\lambda_\ell\, v}\right), \quad \lambda_\ell \geq 0.
$$
We denote the vector of so--called ``normal modes'',
$0 \leq {\lambda}:=\big({\lambda_1}, \ldots, {\lambda_{K-1}}\big)^\top$, and
the matrix of ``Case's eigenfunction''
$$
\Phi^\pm(x) = \begin{pmatrix}
\displaystyle \frac{ e^{-\lambda\, x/\eps}}{1 \mp \calV\otimes \lambda} \quad & \displaystyle \mathbf{1}_{\RR^K}
\quad & \displaystyle \frac{ e^{\lambda\, (x-\Delta x)/\eps}}{1 \pm \calV\otimes \lambda} 
\quad & \displaystyle x \mathbf{1}_{\RR^{K}} \mp \eps \calV
\end{pmatrix} \in \calM_{K\times 2K}(\RR),
$$
such that $\barf(x,\pm\calV)=\Phi^\pm(x)\begin{pmatrix} a \\ b
\end{pmatrix}$, with the notations $a=(a_1,\ldots,a_{K-1},a_0)^\top$ and 
$b=(b_1,\ldots,b_{K-1},b_0)^\top$.
(We recall that the notation $\frac{1}{1+\calV\otimes\lambda}$ denotes the
matrix in $\calM_K(\RR)$ whose coefficients are $(\frac{1}{1+ v_k \lambda_\ell})_{k,\ell}$).
Thanks to the boundary conditions \eqref{eq:stat2} we have
$$
\begin{pmatrix} \barf(0,\calV) \\ \barf(\Delta x,-\calV) \end{pmatrix}
= \begin{pmatrix} \Phi^+(0) \\ \Phi^-(\Delta x) \end{pmatrix}
\begin{pmatrix} a \\ b \end{pmatrix}, \qquad
\begin{pmatrix} \barf(\Delta x,\calV) \\ \barf(0,-\calV) \end{pmatrix}
= \begin{pmatrix} \Phi^+(\Delta x) \\ \Phi^-(0) \end{pmatrix}
\begin{pmatrix} a \\ b \end{pmatrix}.
$$
%
Thus, we deduce (see also \cite[Chap. 9 pp.175-176]{book}), that the solution of 
(\ref{eq:stat1})--(\ref{eq:stat2}) is expressed as,
$$
\begin{pmatrix} \barf_{j-1/2}(\calV) \\ \barf_{j-1/2}(-\calV) \end{pmatrix} =
\Mtilde_\eps M_\eps^{-1} \begin{pmatrix} f_{j-1}(\calV) \\ f_{j}(-\calV) \end{pmatrix},
$$
where
\begin{equation}\label{MtildeRTE}
\Mtilde_\eps := 
\begin{pmatrix} \Phi^+(\Delta x) \\ \Phi^-(0) \end{pmatrix} = 
\begin{pmatrix}
\frac{ e^{-\lambda\, \Delta x/\eps}}{1 - \calV\otimes \lambda} & \mathbf{1}_{\RR^K}
& \frac{ 1}{1 + \calV\otimes \lambda}  & \Delta x\mathbf{1}_{\RR^K} - \eps\calV  \\
\frac{ 1}{1 + \calV\otimes \lambda}  & \mathbf{1}_{\RR^K} &
\frac{ e^{-\lambda\, \Delta x/\eps}}{1 - \calV\otimes \lambda} &  \eps\calV
\end{pmatrix},
\end{equation}
\begin{equation}\label{MRTE}
M_\eps := 
\begin{pmatrix} \Phi^+(0) \\ \Phi^-(\Delta x) \end{pmatrix} =
\begin{pmatrix}
\frac{1}{1 - \calV\otimes \lambda} & \mathbf{1}_{\RR^K}
& \frac{ e^{-\lambda\, \Delta x/\eps}}{1 + \calV\otimes \lambda}  & - \eps\calV  \\
\frac{ e^{-\lambda\, \Delta x/\eps}}{1 + \calV\otimes \lambda} & \mathbf{1}_{\RR^K} &
\frac{ 1}{1 - \calV\otimes \lambda} &  \Delta x\mathbf{1}_{\RR^K}+\eps\calV
\end{pmatrix}.
\end{equation}
Thus we have obtained the first statement of the following Proposition:
\begin{proposition}\label{lem:scatrad}
The scattering matrix for the radiative transfer system is given by
\begin{equation}\label{scatRTE}
\calS^\eps =  \Mtilde_\eps M_\eps^{-1},\qquad \mbox{(independent of $j$)}
\end{equation}
where $\Mtilde_\eps$ and $M_\eps$ are given in \eqref{MtildeRTE}--\eqref{MRTE}.
It admits the decomposition
\begin{equation}
\label{decomp_RT}
\boxed{\calS^\eps = \begin{pmatrix} \mathbf{0}_K & \mathbf{I}_K - \zeta\gamma
\\ \mathbf{I}_K-\zeta\gamma & \mathbf{0}_K \end{pmatrix} + \eps B_\eps,
\, \mbox{ where } \, B_\eps := \frac{1}{\eps}(A_\eps M_\eps^{-1} - A_0 M_0^{-1}),}
\end{equation}
with $A_\eps=\Mtilde_\eps - \begin{pmatrix} \mathbf{0}_K  & \mathbf{I}_K \\ \mathbf{I}_K & \mathbf{0}_K  \end{pmatrix} M_\eps$,
$$
\zeta = \left(\frac{1}{ 1-\calV\otimes \lambda}\right) - \left(\frac{1}{ 1+\calV\otimes \lambda}\right)\in \calM_{K\times K-1}(\RR),
$$ 
and 
$\gamma\in\calM_{K-1 \times K}(\RR)$, $\beta^\top\in \RR^{K}$ are such that
\begin{align}
\gamma \left(\frac{1}{ 1-\calV\otimes \lambda}\right) = \mathbf{I}_{K-1}, \qquad
&\gamma \mathbf{1}_{\RR^K} = \mathbf{0}_{\RR^{K-1}},  \label{eq:invM1} \\
\beta^{\top} \left(\frac{1}{ 1-\calV\otimes \lambda}\right) = \mathbf{0}_{\RR^{K-1}}^\top, \qquad
&\beta^\top \mathbf{1}_{\RR^K}=1. \label{eq:invM2}
\end{align}
\end{proposition}
%
%
\begin{remark}
The existence of $\gamma$ and $\beta$ is provided by Proposition \ref{prop-zeta-gamma} (iii) in Appendix.
It's good to have an ``intuitive idea'' of the nature of $\zeta, \gamma$ and $\beta$:
\begin{itemize}
\item First, $\zeta: \RR^{K-1} \to \RR^K$ converts a set of $K-1$ spectral coefficients into the restriction to $v \in \calV$ of a kinetic density $F(v)$ having a specific property,
$$\forall v_k \in \calV, \qquad 
F(v_k)=-F(-v_k),\quad \mbox{ so } \quad \sum_{k=1}^K \omega_k (F(v_k)+F(-v_k))=0.
$$
\item Then, $\gamma: \RR^{K} \to \RR^{K-1}$ recovers $K-1$ spectral coefficients corresponding to ``damped modes'' (Knudsen layers) out of $K$ samples of any kinetic density $F(v_k)$. More precisely, if we consider a given steady kinetic density,
$$
\forall v_k \in \calV,\qquad 
G(v_k)=a_0 + \sum^{K-1}_{\ell=1} \frac{a_\ell}{1-v_k\,\lambda_\ell},
$$
then $\gamma\, [G(\calV)] = (a_1, a_2, ..., a_{K-1})$: this is meaningful for $\eps \ll 1$.
\item  Oppositely, $\beta$ detects the ``Maxwellian part'' in the decomposition of $G$, that is, the $a_0$ coefficient, so that $\beta^\top \, [G(\calV)] = a_0$.
\end{itemize}
Yet, consider the product $\zeta \gamma: \RR^K \to \RR^K$, applied to $G(\calV)$. It produces,
$$
(\zeta \gamma[G])(v_k)=\sum^{K-1}_{\ell=1} a_\ell \left( \frac{1}{1-v_k\,\lambda_\ell} -  \frac{1}{1+v_k \lambda_\ell}
\right), \qquad k=1, ..., K,
$$
so that,
$$\forall v_k \in \calV, \qquad
\big(\mathbf{I}_K - (\zeta \gamma)\big)[G](v_k)=a_0 +  \sum^{K-1}_{\ell=1} \frac{a_\ell}{1+v_k\,\lambda_\ell}.
$$
\end{remark}
\begin{proof}\begin{enumerate}
\item Based on the simple case (\ref{S-01}), we define $A^\eps$ such that,
$$\forall \eps>0, \qquad
\calS^\eps = \begin{pmatrix} \mathbf{0}_K  & \mathbf{I}_K \\ \mathbf{I}_K & \mathbf{0}_K  \end{pmatrix} 
+ A_\eps M_\eps^{-1},
$$
which clearly yields:
$$
A_\eps=\Mtilde_\eps - \begin{pmatrix} \mathbf{0}_K  & \mathbf{I}_K \\ \mathbf{I}_K & \mathbf{0}_K  \end{pmatrix} M_\eps
= \begin{pmatrix}
\zeta e^{-\lambda\,\Delta x/\eps} & \mathbf{0}_{\RR^K} & -\zeta & -2\calV \eps  \\
-\zeta & \mathbf{0}_{\RR^K} & \zeta e^{-\lambda\,\Delta x/\eps} & 2\calV \eps
\end{pmatrix},
$$
As $\eps\to 0$, 
\begin{equation}\label{eq:A0}
A_\eps \to A_0 = \begin{pmatrix}
 \mathbf{0}_{K} & (-\zeta \quad \mathbf{0}_{\RR^K})  \\
(-\zeta \quad \mathbf{0}_{\RR^K}) &  \mathbf{0}_{K}
\end{pmatrix},
\end{equation}
along with,
\begin{align*}
M_\eps\to M_0 &=
\begin{pmatrix}
\frac{1}{1 - \calV\otimes \lambda} & \mathbf{1}_{\RR^K}
& \mathbf{0}_{K\times (K-1)}  & \mathbf{0}_K  \\
\mathbf{0}_{K\times (K-1)}  & \mathbf{1}_{\RR^K} &
\frac{1}{1 - \calV\otimes \lambda} &  \Delta x\mathbf{1}_{\RR^K}
\end{pmatrix}\\
& := \begin{pmatrix}
M_{01} &\quad  \mathbf{0}_K  \\
(\mathbf{0}_{K\times (K-1)} \quad  \mathbf{1}_{\RR^K}) &\quad M_{02}
\end{pmatrix},
\end{align*}•
with the notation 
$$
M_{01} = \left(\frac{1}{1 - \calV\otimes \lambda} \quad \mathbf{1}_{\RR^K}\right), \qquad
M_{02} = \left(\frac{1}{1 - \calV\otimes \lambda} \quad \Delta x\,\mathbf{1}_{\RR^K}\right).
$$
\item We now intend to compute the following inverse,
\begin{equation}\label{eq:M0}
M_0^{-1} = \begin{pmatrix}
M_{01}^{-1} &  \mathbf{0}_K  \\
-M_{02}^{-1}(\mathbf{0}_{K\times (K-1)} \quad  \mathbf{1}_{\RR^K})M_{01}^{-1}\quad & \quad  M_{02}^{-1}
\end{pmatrix},
\end{equation}
where
$$
M_{01}^{-1} = \begin{pmatrix}
\gamma \\ \beta^\top
\end{pmatrix} \in \calM_{K,K}(\RR),
\qquad 
M_{02}^{-1} = \begin{pmatrix}
\gamma \\ \frac{1}{\Delta x} \beta^\top
\end{pmatrix} \in \calM_{K,K}(\RR),
$$
being $\gamma$ a matrix of size $(K-1)\times K$ and $\beta$ an element in $\RR^K$, such that
\eqref{eq:invM1} and \eqref{eq:invM2} hold.
From the expression \eqref{eq:A0} and \eqref{eq:M0}, it comes:
$$
A_0 M_0^{-1} = \begin{pmatrix}
(-\zeta \quad \mathbf{0}_{\RR^K}) M_{02}^{-1}(\mathbf{0}_{K\times (K-1)} \quad  \mathbf{1}_{\RR^K})M_{01}^{-1}\quad &\quad  (-\zeta \quad \mathbf{0}_{\RR^K}) M_{02}^{-1}  \\
(-\zeta \quad \mathbf{0}_{\RR^K}) M_{01}^{-1} & \mathbf{0}_K
\end{pmatrix}
$$
\begin{itemize}
\item We first observe that
$$
(-\zeta \quad \mathbf{0}_{\RR^K})\begin{pmatrix}
\gamma \\ \beta^\top
\end{pmatrix} = (-\zeta \quad \mathbf{0}_{\RR^K})\begin{pmatrix}
\gamma \\ \frac{1}{\Delta x} \beta^\top
\end{pmatrix} = -\zeta \gamma.
$$
\item Then, using the second identity in \eqref{eq:invM1}, 
$$
A_0 M_0^{-1} = \begin{pmatrix}
\mathbf{0}_K & -\zeta\gamma \\ -\zeta\gamma & \mathbf{0}_K
\end{pmatrix}.
$$
\end{itemize}
Finally, we reach the decomposition (\ref{decomp_RT}).
\end{enumerate}
\qed\end{proof}
\begin{lemma}
With identical notation as Proposition \ref{lem:scatrad}, as $\eps\to 0$,
\begin{equation}\label{eqn_Be}
B_\eps=\frac{1}{\Delta x} \begin{pmatrix} (2\mathbf{I}_K - \zeta\gamma) \calV \beta^\top & 
-(2\mathbf{I}_K - \zeta\gamma) \calV \beta^\top   \\ 
-(2\mathbf{I}_K - \zeta\gamma) \calV \beta^\top  & 
(2\mathbf{I}_K - \zeta\gamma) \calV \beta^\top \end{pmatrix} + o(1).
\end{equation}
\end{lemma}

\begin{proof}
From (\ref{decomp_RT}), and  using the form,
$$
M_\eps = M_0 + \begin{pmatrix}
\mathbf{0}_{K\times(K-1)} & \mathbf{0}_{\RR^K} & \frac{e^{-\lambda\, \Delta x/\eps}}{1+\calV\otimes \lambda} & \quad -\eps\calV \\
 \frac{e^{-\lambda\, \Delta x/\eps}}{1+\calV\otimes \lambda} & \mathbf{0}_{\RR^K} & 
\mathbf{0}_{K\times(K-1)} & \quad \eps \calV
\end{pmatrix},
$$
we get
\begin{align*}
B_\eps & =  \frac{1}{\eps}(A_\eps - A_0 M_0^{-1}M_\eps) M_\eps^{-1} \\
& =  \frac{1}{\eps}(A_\eps-A_0) M_\eps^{-1} \\
& \qquad \qquad +
\begin{pmatrix}
\zeta\gamma(\frac{1}{1+\calV\otimes \lambda}) \delta_\eps \quad \mathbf{0}_{\RR^K} &
\mathbf{0}_{K\times(K-1)} & \zeta\gamma\calV \\
\mathbf{0}_K & \zeta\gamma(\frac{1}{1+\calV\otimes \lambda}) \delta_\eps & -\zeta\gamma\calV 
\end{pmatrix} M_\eps^{-1},
\end{align*}
where $\delta_\eps = \frac{1}{\eps} e^{-\lambda\, \Delta x/\eps} \to 0$ as $\eps\to 0$.
Then,
\begin{equation}
B_\eps = 
\begin{pmatrix}
\zeta\delta_\eps+\zeta\gamma(\frac{1}{1+\calV\otimes \lambda}) \delta_\eps & \mathbf{0}_{\RR^K}
& \mathbf{0}_{K\times(K-1)} & (-2\mathbf{I}_K+\zeta\gamma)\calV \\
\mathbf{0}_{K\times(K-1)} & \mathbf{0}_{\RR^K} & 
\zeta\delta_\eps+\zeta\gamma(\frac{1}{1+\calV\otimes \lambda}) \delta_\eps & (2\mathbf{I}_K-\zeta\gamma)\calV 
\end{pmatrix} M_\eps^{-1}.
\label{scheme1:mat}
\end{equation}
As $\eps\to 0$, we get from \eqref{scheme1:mat} 
$$
B_\eps=\frac{1}{\eps}(A_\eps M_\eps^{-1} - A_0 M_0^{-1}) \underset{\eps\to 0}{\longrightarrow}
B_0 := \begin{pmatrix}
\mathbf{0}_{K\times(2K-1)} & (-2\mathbf{I}_{K}+\zeta\gamma)\calV \\
\mathbf{0}_{K\times(2K-1)} & (2\mathbf{I}_{K}-\zeta\gamma)\calV 
\end{pmatrix} M_0^{-1}.
$$
With the expression of the inverse of $M_0$ in \eqref{eq:M0}, we get
$$
B_0 = \frac{1}{\Delta x} \begin{pmatrix}
(2\mathbf{I}_K - \zeta\gamma) \calV \beta^\top \mathbf{1}_{\RR^K} \beta^\top & -(2\mathbf{I}_K - \zeta\gamma) \calV \beta^\top  \\
-(2\mathbf{I}_K - \zeta\gamma) \calV \beta^\top \mathbf{1}_{\RR^K} \beta^\top & (2\mathbf{I}_K - \zeta\gamma) \calV \beta^\top 
\end{pmatrix}.
$$
Thanks to \eqref{eq:invM2}, we have $\beta^\top \mathbf{1}_{\RR^K}=1$ and we are done.
\qed\end{proof}

\subsection{Emergence of the macroscopic discretization}

Thanks to Proposition \ref{lem:scatrad}, we have at hand the expansion of the scattering matrix
with respect to $\eps$;
then, we may write the final scheme \eqref{schemeAPWB} 
for the radiative transfer equation,
\begin{align}
\begin{pmatrix} f_j^{n+1}(\calV) \\ f_{j-1}^{n+1}(-\calV) \end{pmatrix} 
+\frac{\Delta t}{\eps\Delta x}\VV 
\begin{pmatrix} f_j^{n+1}(\calV)-(\mathbf{I}_K-\zeta\gamma) f_j^{n+1}(-\calV) \\ 
f_{j-1}^{n+1}(-\calV)-(\mathbf{I}_K-\zeta\gamma) f_{j-1}^{n+1}(\calV) \end{pmatrix}
\nonumber \\ =
\begin{pmatrix} f_j^{n}(\calV) \\ f_{j-1}^{n}(-\calV) \end{pmatrix}
+\frac{\Delta t}{\Delta x} \VV B_\eps
\begin{pmatrix} f_{j-1}^{n}(\calV) \\ f_j^{n}(-\calV) \end{pmatrix}.
\label{scheme1:tr}
\end{align}
We are now in position to state our main result for the radiative transfer equation,
whose proof is postponed to the next subsection.
\begin{theorem}\label{th:RTE}
The scheme \eqref{scheme1:tr} for the radiative transfer equation \eqref{eq:tr} is well-balanced
and uniformly accurate (AP) with respect to $\eps$.
Moreover, assuming \eqref{eq:restrict} holds, when $\eps\to 0$, the macroscopic density $\rho^n_j:= \sum_{k=-K}^K \omega_k f_j^n(v_k)$
solves the centered finite difference scheme for the heat equation \eqref{hit_eq}.
\end{theorem}

As notice in \S\ref{sec:idea}, we may rewrite \eqref{scheme1:tr} in IMEX form (see \eqref{IMEX}), 
\begin{equation}\label{scheme2:tr}
\frac{1}{\eps}\calH_\eps 
\begin{pmatrix} f_j^{n+1}(\calV) \\ f_j^{n+1}(-\calV) \end{pmatrix} =
\begin{pmatrix} f_j^{n}(\calV)  \\ f_{j}^n(-\calV)) \end{pmatrix}
+ \frac{\Delta t}{\Delta x} \VV
\begin{pmatrix} B_{1\eps} f_{j-1}^n(\calV)+B_{2\eps} f_j^n(-\calV) \\
B_{3\eps} f_{j}^n(\calV)+B_{4\eps} f_{j+1}^n(-\calV) \end{pmatrix},
\end{equation}
where we use the notation $B_\eps = \begin{pmatrix} B_{1\eps} & B_{2\eps} \\ B_{3\eps} & B_{4\eps} \end{pmatrix}$, and
\begin{equation}\label{RepsRTE}
\calH_\eps = \eps \mathbf{I}_{2N} + \frac{\Delta t}{\Delta x}\VV
\begin{pmatrix} \mathbf{I}_K & \zeta\gamma-\mathbf{I}_K \\ 
\zeta\gamma-\mathbf{I}_K & \mathbf{I}_K \end{pmatrix}.
\end{equation}
Solving this system amounts to the inversion of matrix $\calH_\eps$.
By construction, this scheme satisfies the well-balanced property and is asymptotic preserving.
\begin{proposition}\label{corol1}
Assume that for all discrete eigenvalues $\lambda_i$, $i=1,\ldots,K-1$, 
identity \eqref{ortho0discret} holds, i.e.
$$
\sum_{k=1}^{K} \omega_k v_k \left(\frac{1}{1-v_k\, \lambda_i} - \frac{1}{1+v_k\, \lambda_i} \right) = 0.
$$
Then the scheme (\ref{scheme2:tr}) is mass-preserving, uniformly when $\eps\ll 1$.
\end{proposition}
%
%
\begin{proof}
It is mostly a consequence of Lemma \ref{lem:CGT} and Proposition \ref{lem:scatrad}. Indeed, 
by Lemma \ref{lem:CGT}, $\Gamma \calS^\eps \Gamma^{-1}$ needs to be left stochastic for any $\eps>0$, which implies (by Proposition \ref{lem:scatrad}) that, for $\eps\ll 1$, the matrix
$$
\Gamma \begin{pmatrix} \mathbf{0}_K & \mathbf{I}_K - \zeta\gamma \\ \mathbf{I}_K - \zeta\gamma & 
\mathbf{0}_K \end{pmatrix} \Gamma^{-1} \quad \mbox{ is left stochastic.}
$$
Hence, for all $j=1,\ldots,K$, the sum of the $j^{th}$ column is equal to $1$, that is
$\sum_{k=1}^K \omega_k v_k (\zeta\gamma)_{kj} = 0$.
Accordingly, it comes (by general assumptions) that, 
\begin{align*}
\sum_{k=1}^K \omega_k v_k (\zeta\gamma)_{kj} &= \sum_{k=1}^K \sum_{i=1}^{K-1} \omega_k v_k \left(\frac{1}{1-v_k\, \lambda_i} - \frac{1}{1+v_k\, \lambda_i} \right) \gamma_{ij} = 0.
\end{align*}
\qed\end{proof}

\subsection{Consistency with the diffusive limit}
\label{sec:consistancetr}

This subsection is devoted to the proof of Theorem \ref{th:RTE}.

\noindent{\it Proof of Theorem \ref{th:RTE}}
When $\eps\to 0$, we get from \eqref{scheme1:mat} 
$$
B_\eps
\longrightarrow B_0 = \frac{1}{\Delta x} \begin{pmatrix}
(2\mathbf{I}_K - \zeta\gamma) \calV \beta^\top & -(2\mathbf{I}_K - \zeta\gamma) \calV \beta^\top  \\
-(2\mathbf{I}_K - \zeta\gamma) \calV \beta^\top & (2\mathbf{I}_K - \zeta\gamma) \calV \beta^\top 
\end{pmatrix}.
$$
Moreover, with \eqref{RepsRTE},
$$
\calH_\eps = \calH_0 + \eps \mathbf{I}_{2K}, \quad \mbox{ where }
\calH_0 := \frac{\Delta t}{\Delta x} \VV
\begin{pmatrix} \mathbf{I}_K & \zeta\gamma - \mathbf{I}_K \\ 
\zeta\gamma - \mathbf{I}_K & \mathbf{I}_K \end{pmatrix}.
$$
Assuming that $f$ admits a Hilbert expansion $f=f^0 + \eps f^1 + o(\eps)$.
We inject into scheme \eqref{scheme2:tr}. Then by identifying the term in power of $\eps$, we get
\begin{equation}\label{eqf0}
\calH_0
\begin{pmatrix} \{f^0\}_j^{n+1}(\calV) \\ \{f^0\}_j^{n+1}(-\calV) \end{pmatrix} = 0,
\end{equation}
%
and at order $0$ in $\eps$,
\begin{align}\label{eqf1}
\calH_0
\begin{pmatrix} \{f^1\}_j^{n+1}(\calV) \\ \{f^1\}_j^{n+1}(-\calV) \end{pmatrix}
= \begin{pmatrix}
\{f^0\}_j^{n}(\calV) - \{f^0\}_j^{n+1}(\calV) \\
\{f^0\}_j^{n}(-\calV) - \{f^0\}_j^{n+1}(-\calV) 
\end{pmatrix}
\\ + \frac{\Delta t}{2\Delta x^2} \VV 
\begin{pmatrix}
(2\mathbf{I}_K - \zeta\gamma) \calV \beta^\top (\{f^0\}_{j-1}^n(\calV) - \{f^0\}_j^n(-\calV)) \\ 
(2\mathbf{I}_K - \zeta\gamma) \calV \beta^\top (\{f^0\}_{j+1}^n(\calV) - \{f^0\}_{j}^n(-\calV))
\end{pmatrix}. \nonumber
\end{align}
We will make use of Lemma \ref{Hzero} in the Appendix, which may be applied with $\mu=\lambda$ 
since \eqref{ortho0discret} holds, it gives:
\begin{itemize}
\item Ker$(\calH_0)= \mbox{Span}(\mathbf{1}_{\RR^{2K}})$,
\item Im$(\calH_0)=
\Big\{Z= (Z_1\ Z_2)^\top,\ Z_i\in \RR^{K} \mbox{ such that } 
\sum_{k=1}^K \omega_k ({Z_1}_k+{Z_2}_{k}) = 0\Big\}$.
\end{itemize}
Roughly speaking, the range of $\calH_0$ is an hyperplane containing kinetic densities which moment of order zero vanishes. Its kernel is the (one-dimensional) vectorial line of constant kinetic densities, which are the Maxwellians for (\ref{eq:tr}).
%
We deduce from \eqref{eqf0}
\begin{equation}\label{equilib}
\{f^0\}_j^{n+1}(\pm \calV) = \frac{\rho_j^{n+1}}{2} \mathbf{1}_{\RR^K},
\qquad \rho_j^n=\sum_{k=1}^{K} \omega_k (f_j^n(v_k)+f_j^n(-v_k)).
\end{equation}
Then, injecting \eqref{equilib} into \eqref{eqf1} and using \eqref{eq:invM2},
we obtain
$$
\calH_0
\begin{pmatrix} \{f^1\}_j^{n+1}(\calV) \\ \{f^1\}_j^{n+1}(-\calV) \end{pmatrix}
= \frac{1}{2} \begin{pmatrix}
(\rho_j^{n} - \rho_j^{n+1})\mathbf{1}_{\RR^K} \\
(\rho_j^{n} - \rho_j^{n+1})\mathbf{1}_{\RR^K}
\end{pmatrix}
+ \frac{\Delta t}{\Delta x^2} \VV 
\begin{pmatrix}
(2\mathbf{I}_K - \zeta\gamma) \calV (\rho_{j-1}^n - \rho_j^n)  \\
(2\mathbf{I}_K - \zeta\gamma) \calV (\rho_{j+1}^n - \rho_{j}^n)
\end{pmatrix}.
$$
Moreover, by definition of $\calH_0$, we have
$$
\frac{1}{2\Delta x}\calH_0
\begin{pmatrix} \calV (\rho_{j+1}^n - \rho_j^n) \\ \calV (\rho_{j-1}^n - \rho_j^n)  \end{pmatrix} =
\frac{\Delta t}{2\Delta x^2} \VV
\begin{pmatrix} \calV (\rho_{j+1}^n - \rho_j^n) + (\zeta\gamma -\mathbf{I}_K) \calV (\rho_{j-1}^n - \rho_j^n) \\ 
(\zeta\gamma - \mathbf{I}_K ) \calV (\rho_{j+1}^n - \rho_j^n) + \calV (\rho_{j-1}^n - \rho_j^n)  \end{pmatrix}.
$$
Therefore, adding the last two equalities, we deduce that
\begin{align*}
&\calH_0
\begin{pmatrix} \{f^1\}_j^{n+1}(\calV) - \frac{1}{2\Delta x} \calV (\rho_{j+1}^n - \rho_j^n)  \\
\{f^1\}_j^{n+1}(-\calV) - \frac{1}{2\Delta x} \calV (\rho_{j-1}^n - \rho_j^n) \end{pmatrix}
\\
&= \frac{1}{2} 
(\rho_j^{n} - \rho_j^{n+1})\mathbf{1}_{\RR^{2N}} + \frac{\Delta t}{2\Delta x^2} \VV^2 (\rho_{j+1}^n + \rho_{j-1}^n - 2\rho_j^n).
\end{align*}
A solution exists iff the right hand side belongs to Im$(\calH_0)$, so by Lemma \ref{Hzero},
\begin{equation}\label{rho-FD}
\boxed{
0 = \rho_j^n - \rho_j^{n+1} + \frac{\Delta t}{\Delta x^2} 
(\rho_{j-1}^n - 2 \rho_j^n + \rho_{j+1}^n) \sum_{k=1}^{K} 
\omega_k v_k^2.}
\end{equation}
We conclude the proof thanks to \eqref{eq:restrict}.
%
\qed

\section{Othmer-Alt model for one-dimensional chemotaxis}\label{sec:5}

\subsection{Continuous diffusive limit toward Keller-Segel}

We consider now a model of chemotaxis in parabolic scaling, (see also \cite{emako})
\begin{equation}\label{eq:kinchemo}
\eps \pa_t f + v\pa_x f = \frac{1}{\eps}\left(\int_V T_\eps(t,x,v')f(t,x,v')\,\frac{dv'}{2} - T_\eps(t,x,v)f(t,x,v)\right).
\end{equation}
The tumbling rate $T_\eps$ describes the response to variations of chemical concentration along the path of bacteria. Among several choices,\cite{CGT,othill}, we choose
\begin{equation}\label{eq:Teps}
T_\eps(t,x,v) = 1 + \eps \phi \big(v \pa_x S(t,x) \big), \quad \phi \mbox{ an odd function}.
\end{equation}
In applications, the quantity $S$ models the concentration of the chemoattractant which is released by 
bacteria themselves. It is then computed thanks to an elliptic/parabolic equation depending on the 
density of bacteria. Since we only focus on the diffusive limit of the kinetic system, we will consider
that $S$ is given, which boils down, from a numerical point of view, to treat explicitely in time the
equation for $S$.

Formally, the limit $\eps\to 0$ may be obtained easily by performing a Hilbert expansion, $f=f^0 + \eps f^1 + \ldots$,
equating each term in power of $\eps$ in \eqref{eq:kinchemo}, 
\begin{align*}
&f^0 = \frac{1}{2} \rho^0, \qquad \rho^0 = \int_V f^0(v)\,dv;  \\
&f^1-\frac 12 \int_V f^1(v')\,dv' =\frac 12 \int_V  \phi(v'\pa_xS) f_0(v') \,dv' -\phi(v\pa_xS) f_0 - v\pa_x f^0.
\end{align*}
Since $\phi$ is odd, $V=(-1,1)$ is symmetric and $f^0$ is independent of $v$,  the first term of the right hand side vanishes. By conservation, we have
$$
\pa_t \int_V f^0(v)\,dv + \pa_x \int_V vf^1(v)\,dv = 0.
$$
Along with the expressions of $f^0$ and $f^1$, 
\begin{equation}\label{drift_D}
\pa_t \rho^0 - \pa_x \left( \frac{1}{3} \pa_x \rho^0 + E \rho^0 \right) = 0, \qquad E=\frac 12\int_V v\phi(v\pa_x S)\,dv.
\end{equation}

\subsubsection*{Sharfetter-Gummel scheme}

The Il'in/Sharfetter-Gummel scheme \eqref{SGIntro}--\eqref{FluxSGIntro} for this latter equation reads
\begin{equation}\label{SGchemo}
\rho_j^{n+1} = \rho_j^n + \frac{\Delta t}{\Delta x}
\big(\bar{\calJ}_{j-\frac 12}^n-\bar{\calJ}_{j+\frac 12}^n\big), \quad
\bar{\calJ}_{j-1/2}^n = E_{j-\frac 12} \frac{e^{3E_{j-\frac 12}\Delta x}\rho_j^n-\rho_{j-1}^n}
{1-e^{3E_{j-\frac 12}\Delta x}},
\end{equation}
where $E_{j-\frac 12}$ is a discretization of $E$ at each interface of the mesh.
\begin{remark}
Clearly, if $\phi \equiv 0$, the former model (\ref{eq:tr}) is recovered out of (\ref{eq:kinchemo}), along with its limit (\ref{hit_eq}), being a particular case of (\ref{drift_D}). Accordingly, (\ref{rho-FD}) appears as a restriction of (\ref{SGchemo}) when $E_{j-\frac 12} \equiv 0$. However, the situation $\phi \not \equiv 0$ gives rise to sufficiently strong peculiarities so that we choose, in this paper, to neatly distinguish between both cases.
\end{remark}

As in the previous Section, $V=(-1,1)$, the set $\{\omega_k,v_k\}_k$ is assumed to verify \eqref{eq:restrict}.

\subsection{Asymptotic expansion of eigenvalues}

Let us focus on the eigenvalues of the discrete problem, computed
thanks to the condition in \eqref{eqphilam}, which, for  
$\nu=\frac 12 \sum_{k=-K}^K \omega_k \delta(v-v_k)$, reads
\begin{equation}\label{lambda}
1 = \frac 12 \sum_{k=-K}^K \frac{\omega_k T_\eps(v_k)}{T_\eps(v_k)-\lambda^\eps v_k}.
\end{equation}
Clearly, $\lambda=0$ is a solution.
Thanks to \eqref{eq:restrict}, we deduce from the latter equality that nonzero eigenvalues verify
$$
\sum_{k=-N}^N \frac{\omega_k}{\frac{T_\eps(v_k)}{v_k}-\lambda^\eps} = 0, \quad \mbox{ for } \lambda^\eps\neq 0.
$$
By studying the variations of the left-hand side with respect to $\lambda^\eps$, one deduces the 
existence of exactly $2K-1$ distinct solutions which are interlaced in the following way
\begin{align*}
&\frac{T_\eps(v_{-1})}{v_{-1}} < \lambda_{-K+1}^\eps < \frac{T_\eps(v_{-2})}{v_{-2}} < \ldots < \frac{T_\eps(v_{-K})}{v_{-K}} < \lambda_0^\eps \\
&\qquad <\frac{T_\eps(v_K)}{v_K} < \lambda_1^\eps < \frac{T_\eps(v_{K-1})}{v_{K-1}} < \lambda_2^\eps < \ldots 
< \lambda_{K-1}^\eps < \frac{T_\eps(v_1)}{v_1}.
\end{align*}
The sign of $\lambda_0^\eps$ is given by the sign of 
$\displaystyle \sum_{k=-K}^K \omega_k \frac{v_k}{T_\eps(v_k)}$.

In the following, we always assume that
$\lambda_0^\eps <0$,
the opposite case can be treated in the same way.
Then, the vectors of negative/positive eigenvalues are,
$$
\lambda_-^\eps = (\lambda_{-K+1}^\eps,\ldots,\lambda_{-1}^\eps)^\top,
\qquad 
\lambda_+^\eps = (\lambda_1^\eps,\ldots,\lambda_{K-1}^\eps)^\top.
$$ 
\begin{lemma}\label{lem:lambda}
When $\eps\to 0$, we have 
$\lambda_\ell = \lambda_\ell^0 + \eps \lambda_\ell^1 + o(\eps)$ where 
$$
\lambda_0^0 = 0, \quad \lambda_0^1 = 3 \sum_{k=1}^K \omega_k v_k\phi(v_k\pa_xS).
$$
For $\ell\neq 0$, $\lambda_\ell^0$ are the symmetric $(\lambda^0_{-\ell}=-\lambda^0_\ell)$ eigenvalues of 
$$
\ds 1 = \frac 12 \sum_{k=-K}^K \frac{\omega_k}{1-\lambda^0_\ell v_k},
$$
and
$$
\lambda_\ell^1  \sum_{k=-K}^K \frac{\omega_k v_k}{(1-\lambda^0_\ell v_k)^2}
= \lambda_\ell^0 \sum_{k=-K}^K \frac{\omega_k v_k\phi(v_k\pa_xS)}{(1-\lambda^0_\ell v_k)^2}.
$$
We denote by 
$$
\lambda^0 = (\lambda_1^0,\ldots,\lambda^0_{K-1})^\top
$$
the vector of positive eigenvalues at the limit $\eps\to 0$.
\end{lemma}
\begin{proof}
Letting $\eps \to 0$ in \eqref{lambda}, eigenvalues $\lambda^0$ are solutions of
$$
1 = \frac 12 \sum_{k=-K}^K \frac{\omega_k}{1-\lambda^0 v_k}.
$$
Symmetry of the set $\{\omega_k,v_k\}_k$ implies that if $\lambda^0$ is  solution, then $-\lambda^0$ is, too. As a consequence $\lambda^0_+ = -\lambda^0_-$, and $\lambda^0_0 = 0$.
Assuming an asymptotic expansion $\lambda_k^\eps = \lambda_k^0 + \eps \lambda_k^1 + o(\eps)$, and expanding the relation \eqref{lambda} with \eqref{eq:Teps}, we get
\begin{align*}
1 & = \frac 12 \sum_{k=-K}^K \frac{\omega_k}{1-\lambda^0 v_k- \eps v_k(\lambda^1-\phi(v_k\pa_xS) \lambda^0)+o(\eps)}  \\
&= \frac 12 \sum_{k=-K}^K \frac{\omega_k}{1-\lambda^0 v_k} \left(1 + \eps \frac{v_k(\lambda^1-\phi(v_k\pa_xS) \lambda^0)}{1-\lambda^0 v_k}+o(\eps)\right).
\end{align*}
Thus, for $\ell\neq 0$, 
$$
\lambda_\ell^1 \sum_{k=-K}^K \frac{\omega_k v_k}{(1-\lambda^0_\ell v_k)^2} = 
\lambda_\ell^0 \sum_{k=-K}^K \frac{\omega_k v_k\phi(v_k\pa_xS)}{(1-\lambda^0_\ell v_k)^2}.
$$
For $\ell=0$, this relation gives $\lambda_0^0 = 0$, forcing us to go at the second order in $\eps$ to compute $\lambda^1_0$: postulating that  $\lambda_0^\eps = \eps \lambda_0^{1}+\eps^2 \lambda_0^2 + \ldots$, it comes
\begin{align*}
1 & = \frac 12 \sum_{k=-K}^K \frac{\omega_k}{1- \eps v_k(\lambda^{1}_0+\eps\lambda^2_0)/(1+\eps\phi(v_k\pa_xS))}  \\
&= \frac 12 \sum_{k=-K}^K \omega_k \left(1 + \eps v_k \frac{\lambda_0^{1}+\eps\lambda_0^2}{1+\eps\phi(v_k\pa_xS)}+\eps^2 v_k^2 (\lambda_0^1)^2 +o(\eps^2)\right).
\end{align*}
We get
$$
0 = \sum_{k=-K}^K \omega_k v_k\lambda_0^1 + \eps \sum_{k=-K}^K \omega_k \Big(-v_k\phi(v_k\pa_xS) \lambda_0^1 + v_k \lambda_0^2 + v_k^2(\lambda_0^1)^2\Big) + o(\eps).
$$
By symmetry of the set $\{\omega_k,v_k\}$, we have $\sum_{k=-K}^K \omega_k v_k=0$.
Then, assumptions (\ref{eq:restrict}) yield two solutions, among which we discard the null one, it gives
$$
\lambda_0^1 = \frac{3}{2} \sum_{k=-K}^K \omega_k v_k\phi(v_k\pa_x S). 
$$
Being $\phi$ an odd function and by symmetry of $\{\omega_k,v_k\}_k$, the claim is proved.
\qed\end{proof}

\subsection{Corresponding scattering $S$-matrix}

A general stationary solution reads, (see the separation of variables in \cite{vincent})
\begin{align*}
\barf^\eps(x,v) &= \sum_{\ell=1}^{K-1} \frac{a_\ell e^{-\lambda_\ell^\eps x/\eps}}{T_\eps(v)-\lambda_\ell^\eps v} 
+ \frac{\overline{a}}{T_\eps(v)} \\
& \qquad + \sum_{\ell=-K+1}^{-1} \frac{a_\ell e^{-\lambda_\ell^\eps (x-\Delta x)/\eps}}{T_\eps(v)-\lambda_\ell^\eps v}
+ a_0 \left(\frac{e^{-\lambda_0^\eps \frac{x}{\eps}}}{T_\eps(v)-\lambda_0^\eps v} - \frac{1}{T_\eps(v)}\right).
\end{align*}•
The spectral component of the ``zero-eigenfunction'' $1/T_\eps(v)$ was split between $\bar a$ and $a_0$. Like for radiative transfer (see \S\ref{sec:SRTE}), where $\phi \equiv 0$, the  $S$-matrix is,
\begin{equation}\label{def:scat}
\calS^\eps = \widetilde{N^\eps} (N^\eps)^{-1}, \qquad \mbox{ (dependent of $j-\frac 1 2$)}
\end{equation}
where the $j-\frac 1 2$ index was dropped for the sake of simplicity of the scripture,
\begin{align*}
N^\eps = 
\begin{pmatrix} 
\ds \frac{1}{T_\eps(\calV) - \calV\otimes \lambda_+} & 
\ds \frac{1}{T_\eps(\calV)} &
\ds \frac{e^{\lambda_-\Delta x/\eps}}{T_\eps(\calV)-\calV\otimes \lambda_-} &
\ds \frac{1}{T_\eps(\calV)-\lambda_0^\eps \calV} - \frac{1}{T_\eps(\calV)}  \\
\ds \frac{e^{-\lambda_+\Delta x/\eps}}{T_\eps(-\calV) + \calV\otimes \lambda_+} &
\ds \frac{1}{T_\eps(-\calV)} &
\ds \frac{1}{T_\eps(-\calV)+\calV\otimes \lambda_-} &
\ds \frac{e^{-\lambda_0^\eps \Delta x/\eps}}{T_\eps(-\calV)+\lambda_0^\eps \calV} - \frac{1}{T_\eps(-\calV)}
\end{pmatrix}, \\
\widetilde{N^\eps} = 
\begin{pmatrix} 
\ds \frac{e^{-\lambda_+\Delta x/\eps}}{T_\eps(\calV) - \calV\otimes \lambda_+} &
\ds  \frac{1}{T_\eps(\calV)} &
\ds \frac{1}{T_\eps(\calV)-\calV\otimes \lambda_-} &
\ds \frac{e^{-\lambda_0^\eps \Delta x/\eps}}{T_\eps(\calV)-\lambda_0^\eps \calV} - \frac{1}{T_\eps(\calV)} \\
\ds \frac{1}{T_\eps(-\calV) + \calV\otimes \lambda_+} & 
\ds \frac{1}{T_\eps(-\calV)} & 
\ds \frac{e^{\lambda_-\Delta x/\eps}}{T_\eps(-\calV)+\calV\otimes \lambda_-}  & 
\ds \frac{1}{T_\eps(-\calV)+\lambda_0^\eps \calV} - \frac{1}{T_\eps(-\calV)}
\end{pmatrix}.
\end{align*}
The main differences with radiative transfer, as in \S\ref{sec:4}, are
that $0$ is a simple eigenvalue but becomes a double eigenvalue when $\eps\to 0$, and 
that all above quantities depend on the index $j$, which lead to more intricate computations.

\subsection{Decomposition of the scattering matrix}

We want to compute the expansion \eqref{decompS} for this scattering matrix. 
\begin{lemma}\label{lem:scatdev}
The scattering matrix for the kinetic model for chemotaxis, defined in \eqref{def:scat} admits the following asymptotic expansion in $\eps$,
\begin{equation}\label{eq:scatdev}
\fbox{$\calS^\eps = \begin{pmatrix} 
\mathbf{0}_K & \mathbf{I}_K - \zeta^0 \gamma  \\
\mathbf{I}_K - \zeta^0 \gamma & \mathbf{0}_K
\end{pmatrix}
+ \eps B^\eps, \quad 
B^\eps = \frac{1}{\eps}(A^\eps (N^\eps)^{-1} - A^0 (N^0)^{-1})$,}
\end{equation}
where $A^\eps=\widetilde{N^\eps} - \begin{pmatrix} \mathbf{0}_K & \mathbf{I}_K \\ \mathbf{I}_K & \mathbf{0}_K \end{pmatrix} N^\eps$,
and where the matrices $\gamma\in \calM_{K-1\times K}(\RR)$ and
$\zeta^0\in \calM_{K\times K-1}(\RR)$ satisfy,
\begin{align}
\label{def:gam}
&\gamma \frac{1}{1-\calV\otimes \lambda^0} = \mathbf{I}_{K-1}, \qquad
\gamma \mathbf{1}_{\RR^K} = \mathbf{0}_{\RR^{K-1}}, \\
\label{def:zeta0}
&\zeta^0 = \frac{1}{1-\calV\otimes \lambda^0} - \frac{1}{1+\calV\otimes \lambda^0}.
\end{align}
\end{lemma}
\begin{remark}
Lemma \ref{lem:scatdev} is the equivalent of Lemma \ref{lem:scatrad} in the case of the radiative transfert equation. 
The existence of $\gamma$ is also provided by Proposition \ref{prop-zeta-gamma} (iii) in Appendix.
\end{remark}
\begin{proof}
The index $j-1/2$ is again dropped since there is no possible confusion.
\begin{enumerate}
\item When $\eps\to 0$, being $T_\eps$ given in \eqref{eq:Teps} and using Lemma \ref{lem:lambda}
$$
\frac{1}{T_\eps(\calV)-\lambda_0^\eps \calV} - \frac{1}{T_\eps(\calV)} \to 0,
\qquad
\frac{e^{-\lambda_0^\eps \Delta x/\eps}}{T_\eps(\calV)-\lambda_0^\eps \calV} - \frac{1}{T_\eps(\calV)} \to e^{-\lambda^1_0 \Delta x}-1.
$$
so that $\calS^\eps \to \calS^0$, where $\calS^0 = \widetilde{N^0} (N^0)^{-1}$. Since $\lambda^0:=\lambda_+^0=-\lambda_-^0$,
\begin{align*}
N^0 = \begin{pmatrix}
\frac{1}{1 - \calV\otimes \lambda^0} & \mathbf{1}_{\RR^K} & \mathbf{0}_K \\
\mathbf{0}_{K\times K-1} & \mathbf{1}_{\RR^K} & \Big(\frac{1}{1-\lambda^0\otimes \calV} \quad (e^{-\lambda^1_0 \Delta x}-1)\mathbf{1}_{\RR^K}\Big)
\end{pmatrix}, \\
\widetilde{N^0} = \begin{pmatrix}
\mathbf{0}_{K\times K-1} & \mathbf{1}_{\RR^K} & 
\Big(\frac{1}{1 + \calV\otimes \lambda^0} \quad (e^{-\lambda^1_0 \Delta x}-1)\mathbf{1}_{\RR^K} \Big)\\
\frac{1}{1+\lambda^0\otimes \calV} & \mathbf{1}_{\RR^K} & \mathbf{0}_K
\end{pmatrix}.
\end{align*}
\item Using $\gamma$ defined in \eqref{def:gam}, we define also 
$\beta \in \RR^K$ such that
\begin{equation}
\label{def:bet}
\beta^\top \frac{1}{1-\calV\otimes \lambda^0} = \mathbf{0}_{\RR^K}^\top, \qquad
\beta^\top \mathbf{1}_{\RR^K} = 1.
\end{equation}
Then, we may write
\begin{align*}
(N^0)^{-1} = &
\begin{pmatrix}
\begin{pmatrix} \gamma\\ \beta^\top \end{pmatrix} & \mathbf{0}_K  \\
- \begin{pmatrix} \gamma\\ \frac{1}{e^{-\lambda_0^1\Delta x}-1}\beta^\top \end{pmatrix}\Big(\mathbf{0}_{K\times K-1} \quad \mathbf{1}_{\RR^K}\Big) \begin{pmatrix} \gamma \\ \beta^\top \end{pmatrix} &
\begin{pmatrix} \gamma\\ \frac{1}{e^{-\lambda_0^1\Delta x}-1}\beta^\top \end{pmatrix}
\end{pmatrix}  \\
= &\begin{pmatrix}
\begin{pmatrix} \gamma\\ \beta^\top \end{pmatrix} & \mathbf{0}_K  \\
- \begin{pmatrix} \mathbf{0}_{K-1\times K} \\ \frac{1}{e^{-\lambda_0^1\Delta x}-1}\beta^\top \end{pmatrix} &
\begin{pmatrix} \gamma\\ \frac{1}{e^{-\lambda_0^1\Delta x}-1}\beta^\top \end{pmatrix}
\end{pmatrix}.
\end{align*}
By definition, 
$$
A^\eps = \widetilde{N^\eps} - \begin{pmatrix} \mathbf{0}_K & \mathbf{I}_K \\ \mathbf{I}_K & \mathbf{0}_K \end{pmatrix} N^\eps,
\quad \mbox{ and so } \quad 
\calS^\eps = \begin{pmatrix}  \mathbf{0}_K & \mathbf{I}_K \\ \mathbf{I}_K & \mathbf{0}_K \end{pmatrix} + A^\eps (N^\eps)^{-1}.
$$
\item Let us denote
\begin{align*}
&\zeta_\pm^\eps = \frac{\pm 1}{T_\eps(\calV)-\calV\otimes\lambda_\pm^\eps} 
- \frac{\pm 1}{T_\eps(-\calV)+\calV\otimes\lambda_\pm^\eps}\in \calM_{K\times K-1}(\RR), \\
&\zeta_0^\eps = \frac{1}{T_\eps(\calV)-\lambda_0^\eps\calV} - \frac{1}{T_\eps(-\calV)+\lambda_0^\eps\calV} \in \RR^K.
\end{align*}
As $\eps\to 0$, we obtain the limit
$$
\zeta_\pm^\eps \to \zeta^0 = \frac{1}{1-\calV\otimes \lambda^0} - \frac{1}{1+\calV\otimes \lambda^0}, \quad \mbox{ and } \quad
\zeta_0^\eps \to 0.
$$
From the expression,
\begin{equation}\label{eq:Aeps}
A^\eps = \begin{pmatrix}
e^{-\lambda_+ \Delta x /\eps}\zeta_+^\eps & \frac{1}{T_\eps(\calV)} - \frac{1}{T_\eps(-\calV)} & -\zeta_-^\eps
& e^{-\lambda_0^\eps\Delta x/\eps}\zeta_0^\eps - \frac{1}{T_\eps(\calV)}+\frac{1}{T_\eps(-\calV)}  \\
- \zeta_+^\eps & \frac{1}{T_\eps(-\calV)} - \frac{1}{T_\eps(\calV)} & e^{\lambda_-\Delta x/\eps}\zeta_-^\eps & -\zeta_0^\eps+\frac{1}{T_\eps(\calV)}-\frac{1}{T_\eps(-\calV)}
\end{pmatrix}.
\end{equation}
in the limit $\eps\to 0$, we get
\begin{equation}\label{eq:A0N0}
A^0 = \begin{pmatrix}
\mathbf{0}_K & \Big(-\zeta^0 \quad \mathbf{0}_{\RR^K}\Big) 
\\ \Big(-\zeta^0 \quad \mathbf{0}_{\RR^K}\Big) & \mathbf{0}_K
\end{pmatrix}, \qquad 
A^0(N^0)^{-1} = \begin{pmatrix} 
\mathbf{0}_K & -\zeta^0 \gamma  \\
- \zeta^0 \gamma & \mathbf{0}_K
\end{pmatrix}.
\end{equation}
Thus we reach decomposition \eqref{eq:scatdev}.
\end{enumerate}
\qed\end{proof}

\subsection{Emergence of an asymptotic scheme}

We deduce the final scheme from \eqref{schemeAPWB}, which mostly reads as \eqref{scheme1:tr},
\begin{align}
\begin{pmatrix} f_j^{n+1}(\calV) \\ f_{j-1}^{n+1}(-\calV) \end{pmatrix} 
&+\frac{\Delta t}{\eps\Delta x}\VV 
\begin{pmatrix} f_j^{n+1}(\calV)-(\mathbf{I}_K-\zeta^0_{j-\frac 12}\gamma_{j-\frac 12}) f_j^{n+1}(-\calV) \\ 
f_{j-1}^{n+1}(-\calV)-(\mathbf{I}_K-\zeta^0_{j-\frac 12}\gamma_{j-\frac 12}) f_{j-1}^{n+1}(\calV) \end{pmatrix}   \nonumber \\
&=
\begin{pmatrix} f_j^{n}(\calV) \\ f_{j-1}^{n}(-\calV) \end{pmatrix}
+\frac{\Delta t}{\Delta x} \VV B^\eps_{j-\frac 12}
\begin{pmatrix} f_{j-1}^{n}(\calV) \\ f_j^{n}(-\calV) \end{pmatrix}.
\label{scheme1:chemo}
\end{align}
Denoting $B^\eps = \begin{pmatrix} B^{1\eps} & B^{2\eps} \\ B^{3\eps} & B^{4\eps} \end{pmatrix}$,
we may rewrite \eqref{scheme1:chemo} as
\begin{equation}\label{scheme2:chemo}
\frac{1}{\eps}\calH^\eps_j 
\begin{pmatrix} f_j^{n+1}(\calV) \\ f_j^{n+1}(-\calV) \end{pmatrix} =
\begin{pmatrix} f_j^{n}(\calV)  \\ f_{j}^n(-\calV)) \end{pmatrix}
+ \frac{\Delta t}{\Delta x} \VV
\begin{pmatrix} B^{1\eps}_{j-\frac 12} f_{j-1}^n(\calV)+B^{2\eps}_{j-\frac 12} f_j^n(-\calV) \\
B^{3\eps}_{j+\frac 12} f_{j}^n(\calV)+B^{4\eps}_{j+\frac 12} f_{j+1}^n(-\calV) \end{pmatrix},
\end{equation}
where
\begin{equation}\label{Reps:chemo}
\calH^\eps_j = \eps \mathbf{I}_{2K} + \frac{\Delta t}{\Delta x}\VV
\begin{pmatrix} \mathbf{I}_K & \zeta^0_{j-\frac 12}\gamma_{j-\frac 12}-\mathbf{I}_K \\ 
\zeta^0_{j+\frac 12}\gamma_{j+\frac 12}-\mathbf{I}_K & \mathbf{I}_K \end{pmatrix}.
\end{equation}
After inversion of the matrix $\calH_\eps$, 
by construction, this scheme satisfies both the well-balanced property and is asymptotic preserving.

Our main result for the Othmer-Alt model for chemotaxis \eqref{eq:kinchemo} reads:
\begin{theorem}\label{th:chemo}
The scheme \eqref{scheme2:chemo} for the chemotaxis model \eqref{eq:kinchemo} is
well-balanced and uniformly accurate (AP) with respect to $\eps$. 
Moreover, when $\eps\to 0$, the macroscopic density 
$\displaystyle\rho_j^n:=\sum_{k=-K}^K \omega_k f_j^n(v_k)$
satisfies the Sharfetter-Gummel discretization \eqref{SGchemo}, where 
$\displaystyle E_{j-\frac 12}:=\sum_{k=1}^K \omega_k v_k \phi(v_k \pa_xS_{j-\frac 12})$.
\end{theorem}
The proof of this Theorem is done in the next subsection.

\subsection{Consistency with the limit $\eps\to 0$}

We first state the following technical Lemma (dropping the subscript $j-\frac 12$ since there is no possible confusion):
\begin{lemma}\label{lem:limB}
When $\eps$ goes to $0$, 
$$
B^\eps = \begin{pmatrix} B^{1\eps} & B^{2\eps} \\ B^{3\eps} & B^{4\eps} \end{pmatrix} \to 
B^0 = \begin{pmatrix} B^{10} & B^{20} \\ B^{30} & B^{40} \end{pmatrix},
$$
where
$$
\left\{\begin{array}{rcl}
B^{10} &=& \ds -\frac{e^{-\lambda_0^1\Delta x}}{e^{-\lambda_0^1\Delta x}-1}(2-\zeta^0\gamma)\phi(\calV\pa_xS)\beta^\top; \\
B^{20} &=& \ds \left((\zeta^0\gamma - 1)\frac{\phi(\calV\pa_xS)-\calV\otimes \lambda_-^1}{(1-\calV\otimes \lambda^0)^2}
+\frac{\phi(\calV\pa_xS)-\calV\otimes\lambda_-^1}{(1+\calV\otimes \lambda^0)^2}\right)\gamma \\
&& \ds \qquad + \frac{1}{e^{-\lambda_0^1\Delta x}-1}(2-\zeta^0\gamma)\phi(\calV\pa_xS)\beta^\top; \\
B^{30} &=& \ds \left((\zeta^0\gamma - 1)\frac{-\phi(\calV\pa_xS)+\calV\otimes\lambda_+^1}{(1-\calV\otimes \lambda^0)^2}
+\frac{-\phi(\calV\pa_xS)+\calV\otimes\lambda_+^1}{(1+\calV\otimes \lambda^0)^2}\right)\gamma  \\
&& \ds \qquad + (2-\zeta^0\gamma)\phi(\calV\pa_xS)\beta^\top +\frac{1}{e^{-\lambda_0^1\Delta x}-1}(2-\zeta^0\gamma)\lambda_0^1\calV\beta^\top; \\
B^{40} &=& -\ds \frac{1}{e^{-\lambda_0^1\Delta x}-1}(2-\zeta^0\gamma)\lambda_0^1\calV\beta^\top;
\end{array}\right.
$$
with $\gamma$, $\zeta^0$, and $\beta$ defined in \eqref{def:gam}, \eqref{def:zeta0}, and \eqref{def:bet} respectively.
\end{lemma}
\begin{proof}
\begin{enumerate}
\item First, rewrite 
$$
B^\eps = \frac{1}{\eps}(A_\eps-A_0)N_\eps^{-1} + \frac{1}{\eps}(A_0-A_0N_0^{-1}N_\eps)N_\eps^{-1}.
$$ 
From \eqref{eq:Aeps} and \eqref{eq:A0N0}, we have
\begin{align*}
\frac{A_\eps-A_0}{\eps} =
\left(\begin{matrix}
\frac{1}{\eps} e^{-\lambda_+ \Delta x /\eps}\zeta_+^\eps &\quad \frac{1}{\eps}\left(\frac{1}{T_\eps(\calV)} - \frac{1}{T_\eps(-\calV)}\right) &\quad -\frac{1}{\eps}(\zeta_-^\eps-\zeta^0) \quad \\
- \frac{1}{\eps}(\zeta_+^\eps-\zeta^0) &\quad \frac{1}{\eps}\left(\frac{1}{T_\eps(-\calV)} - \frac{1}{T_\eps(\calV)}\right) &\quad \frac{1}{\eps}e^{\lambda_-\Delta x/\eps}\zeta_-^\eps \quad
\end{matrix}\right.  \\
\left.\begin{matrix}
\frac{1}{\eps}\left(e^{-\lambda_0^\eps\Delta x/\eps}\zeta_0^\eps - \frac{1}{T_\eps(\calV)}+\frac{1}{T_\eps(-\calV)}\right)  \\
\frac{1}{\eps}\left(-\zeta_0^\eps+\frac{1}{T_\eps(\calV)}-\frac{1}{T_\eps(-\calV)}\right)
\end{matrix}\right),
\end{align*}
where, from \eqref{eq:Teps},
$$
\lim_{\eps \to 0} \frac{1}{\eps} \left(\frac{1}{T_\eps(\calV)} -\frac{1}{T_\eps(-\calV)}\right) = - 2 \phi(\calV \pa_x S).
$$
\item Denoting $\lambda^1_\pm = \lim_{\eps\to 0} \frac{1}{\eps} (\lambda_\pm^\eps - \lambda^0)$ (as in Lemma \ref{lem:lambda}), we get
$$
\lim_{\eps\to 0} \frac{1}{\eps}(\zeta_\pm^\eps -\zeta^0) = 
\frac{\calV\otimes\lambda_\pm^1-\phi(\calV \pa_xS)}{(1\mp\calV\otimes \lambda^0)^2} -
\frac{\calV\otimes\lambda_\pm^1-\phi(\calV \pa_xS)}{(1\pm\calV\otimes \lambda^0)^2}.
$$
Also, 
$$
\frac{1}{\eps} \zeta_0^\eps \to -2\phi(\calV\pa_xS)+2\lambda_0^1\calV.
$$
Thus, when $\eps\to 0$,
\begin{align*}
\lim_{\eps\to 0} \frac{A_\eps-A_0}{\eps} =
\left( \begin{matrix}
\mathbf{0}_{K\times K-1} &\quad - 2 \phi(\calV \pa_x S)\quad \\
\frac{\phi(\calV \pa_xS)-\calV\otimes\lambda_+^1}{(1-\calV\otimes \lambda^0)^2} - \frac{\phi(\calV \pa_xS)-\calV\otimes\lambda_+^1}{(1+\calV\otimes \lambda^0)^2}
&\quad 2 \phi(\calV \pa_x S)\quad
\end{matrix}\right.    \\
\left.\begin{matrix}
\frac{\phi(\calV \pa_xS)-\calV\otimes\lambda_-^1}{(1+\calV\otimes \lambda^0)^2} - \frac{\phi(\calV \pa_xS)-\calV\otimes\lambda_-^1}{(1-\calV\otimes \lambda^0)^2}
&\quad  2 \phi(\calV \pa_x S)  \\
\mathbf{0}_{K\times K-1} &\quad  - 2 \lambda_0^1 \calV
\end{matrix}\right).
\end{align*}
\item Thanks to the fact that $\gamma \frac{1}{1-\calV\otimes \lambda^0}=\mathbf{I}_{K-1}$, it comes 
\begin{align*}
&\frac{1}{\eps}(A_0-A_0N_0^{-1}N_\eps) = 
\left(\begin{matrix} 
\zeta^0 \frac{\gamma}{\eps}\left(\frac{e^{-\lambda_+\Delta x/\eps}}{T_\eps(-\calV) + \calV\otimes \lambda_+} \quad \frac{1}{T_\eps(-\calV)}\right) \\
\zeta^0 \frac{\gamma}{\eps}\left(\frac{1}{T_\eps(\calV) - \calV\otimes \lambda_+}-\frac{1}{1-\calV\otimes \lambda^0} 
\quad \frac{1}{T_\eps(\calV)}\right)
\end{matrix}\right.  \\
&\qquad\qquad\left.\begin{matrix}
\zeta^0 \frac{\gamma}{\eps}\left(\frac{1}{T_\eps(-\calV)+\calV\otimes \lambda_-} -\frac{1}{1-\calV\otimes \lambda^0} 
\quad \frac{e^{-\lambda_0^\eps \Delta x/\eps}}{T_\eps(-\calV)+\lambda_0^\eps \calV} - \frac{1}{T_\eps(-\calV)}\right)  \\
\zeta^0 \frac{\gamma}{\eps}\left(\frac{e^{\lambda_-\Delta x/\eps}}{T_\eps(\calV)-\calV\otimes \lambda_-} 
\quad \frac{1}{T_\eps(\calV)-\lambda_0^\eps \calV} - \frac{1}{T_\eps(\calV)}\right)
\end{matrix}\right).
\end{align*}
Since $\gamma \mathbf{1}_{\RR^K} = \mathbf{0}_{\RR^{K-1}}$, we have
$$
\gamma \frac{1}{\eps T_\eps(\calV)} = \frac{\gamma}{\eps} \left(\frac{1}{T_\eps(\calV)}-\mathbf{1}_{\RR^K}\right) = \gamma \frac{-\phi(\calV \pa_xS)}{1+\eps\phi(\calV \pa_xS)} \underset{\eps\to 0}{\longrightarrow} - \gamma \phi(\calV \pa_xS).
$$
We have also
$$
\frac{\gamma}{\eps}\left(\frac{1}{T_\eps(\calV) - \calV\otimes \lambda_+}-\frac{1}{1-\calV\otimes \lambda^0}\right) 
\underset{\eps\to 0}{\longrightarrow}
\gamma \frac{-\phi(\calV \pa_x S)+\calV\otimes\lambda^1_+}{(1-\lambda^0\otimes \calV)^2}.
$$
Then,
\begin{align*}
&\lim_{\eps\to 0} \frac{1}{\eps}(A_0-A_0N_0^{-1}N_\eps) =  \\
&\begin{pmatrix} 
  \zeta^0 \gamma \left(\mathbf{0}_{K\times K-1} \quad \phi(\calV \pa_x S)\right) &
\zeta^0 \gamma \left(\frac{\phi(\calV \pa_x S)-\calV\otimes\lambda^1_-}{(1-\lambda^0\otimes \calV)^2}
\quad -\phi(\calV \pa_x S) \right)  \\
\zeta^0 \gamma\left(\frac{-\phi(\calV \pa_x S)+\calV\otimes\lambda^1_+}{(1-\lambda^0\otimes \calV)^2}
\quad -\phi(\calV\pa_xS)\right) &
\zeta^0 \gamma\left(\mathbf{0}_{K\times K-1}
\quad \lambda_0^1 \calV\right)
\end{pmatrix}.
\end{align*}
\item Gathering these computations, we reach
\begin{align*}
B^0 &:=\lim_{\eps\to 0} B^\eps \\
& =\lim_{\eps\to 0} \frac{A^\eps - A^0}{\eps} (N^0)^{-1} +
\lim_{\eps\to 0} \frac{1}{\eps} (A^0-A^0(N^0)^{-1}N^\eps) (N^0)^{-1} \\
&= \begin{pmatrix}
B^{10} & B^{20} \\ B^{30} & B^{40}
\end{pmatrix},
\end{align*}
where  $B^{10}$, $B^{20}$, $B^{30}$, and $B^{40}$ are expressed in Lemma \ref{lem:limB}.
\end{enumerate}
\qed\end{proof}
%
{\it Proof of Theorem \ref{th:chemo}}
We proceed as in Section \ref{sec:consistancetr}. First, we have obviously from \eqref{Reps:chemo}
$\calH^\eps_j=\eps \mathbf{I}_{2K} + \calH^0_j$ where
$$
\calH^0_j := \frac{\Delta t}{\Delta x} \VV
\begin{pmatrix} \mathbf{I}_K & \zeta^0_{j-\frac 12}\gamma _{j-\frac 12}- \mathbf{I}_K \\ 
\zeta^0_{j+\frac 12}\gamma_{j+\frac 12} - \mathbf{I}_K & \mathbf{I}_K \end{pmatrix}
$$
Assuming that $f$ admits an Hilbert expansion $f=f^0+\eps f^1 + \ldots$, we 
get, by injecting this expansion into \eqref{scheme2:chemo} and identifying the
term in power of $\eps$,
\begin{equation}\label{eqf0chemo}
\calH^0_j
\begin{pmatrix} \{f^0\}_j^{n+1}(\calV) \\ \{f^0\}_j^{n+1}(-\calV) \end{pmatrix} = 0,
\end{equation}
and
\begin{align}\label{eqf1chemo}
\calH_j^0
\begin{pmatrix} \{f^1\}_j^{n+1}(\calV) \\ \{f^1\}_j^{n+1}(-\calV) \end{pmatrix}
&= \frac{1}{2} \begin{pmatrix}
\{f^0\}_j^{n} - \{f^0\}_j^{n+1}\\
\{f^0\}_j^{n} - \{f^0\}_j^{n+1}
\end{pmatrix}\\ 
& \qquad + \frac{\Delta t}{2\Delta x} \VV 
\begin{pmatrix} 
B_{j-\frac 12}^{10} \{f^0\}_{j-1}^n +B^{20}_{j-\frac 12} \{f^0\}_j^n \\[1mm]
 B^{30}_{j+\frac 12} \{f^0\}_{j}^n+B^{40}_{j+\frac 12} \{f^0\}_{j+1}^n 
\end{pmatrix}. \nonumber
\end{align}
%
Since from \eqref{ortho0discret} we have that $\forall \ell$,  
$\sum_{k=1}^K \omega_k v_k (\zeta^0 \gamma)_{k\ell} = 0,$
we may apply Lemma \ref{Hzero} and deduce
\begin{itemize}
\item Ker$(\calH_j^0)=$ Span$\{\mathbf{1}_{\RR^K}\}$,
\item Im$(\calH_j^0)=\{Z=(Z_1\quad Z_2)^\top,\ Z_i\in\RR^K$ such that $\sum_{k=1}^K \omega_k({Z_1}_k+{Z_2}_k)=0\}$.
\end{itemize}
Thus equation \eqref{eqf0chemo} implies
$$
\{f^0\}_j^{n+1}(\pm\calV)=\frac{\rho_j^{n+1}}{2} \mathbf{1}_{\RR^K}.
$$
%
Injecting into \eqref{eqf1chemo}, we get
\begin{align}\label{eqf1chemobis}
\calH_j^0
\begin{pmatrix} \{f^1\}_j^{n+1}(\calV) \\ \{f^1\}_j^{n+1}(-\calV) \end{pmatrix}
&= \frac{1}{2} \begin{pmatrix}
(\rho_j^{n} - \rho_j^{n+1}) \mathbf{1}_{\RR^K} \\
(\rho_j^{n} - \rho_j^{n+1}) \mathbf{1}_{\RR^K} 
\end{pmatrix}\\ 
& \qquad + \frac{\Delta t}{2\Delta x} \VV 
\begin{pmatrix} 
B_{j-\frac 12}^{10} \mathbf{1}_{\RR^K} \rho_{j-1}^n +B^{20}_{j-\frac 12} \mathbf{1}_{\RR^K} \rho_j^n \\[1mm]
 B^{30}_{j+\frac 12} \mathbf{1}_{\RR^K} \rho_{j}^n+B^{40}_{j+\frac 12}\mathbf{1}_{\RR^K} \rho_{j+1}^n 
\end{pmatrix}. \nonumber
\end{align}
Thanks to the relations,
$$\gamma \mathbf{1}_{\RR^K} = \mathbf{0}_{\RR^{K-1}}. 
\qquad \beta^\top \mathbf{1}_{\RR^K} = 1,
$$ 
we deduce from the expression of $B^0$ in Lemma \ref{lem:limB},
\begin{align*}
&B^{10}_{j-\frac 12} \mathbf{1}_{\RR^K} = -\frac{e^{-\lambda_{0,j-\frac 12}^1\Delta x}}{e^{-\lambda_{0,j-\frac 12}^1\Delta x}-1}(2-\zeta^0\gamma)\phi(\calV\pa_xS_{j-\frac 12}); \\
&B^{20}_{j-\frac 12} \mathbf{1}_{\RR^K} = \frac{1}{e^{-\lambda_{0,j-\frac 12}^1\Delta x}-1}(2-\zeta^0\gamma)\phi(\calV\pa_xS_{j-\frac 12}); \\
&B_{j+\frac 12}^{30} \mathbf{1}_{\RR^K} = (2-\zeta^0\gamma)\phi(\calV\pa_xS_{j+\frac 12})
 +\frac{\lambda_{0,j+\frac 12}^1}{e^{-\lambda_{0,j+\frac 12}^1\Delta x}-1}(2-\zeta^0\gamma)\calV; \\
&B_{j+\frac 12}^{40} \mathbf{1}_{\RR^K} = \frac{-\lambda_{0,j+\frac 12}^1}{e^{-\lambda_{0,j+\frac 12}^1\Delta x}-1}(2-\zeta^0\gamma)\calV.
\end{align*}
The solvability condition for \eqref{eqf1chemobis} is that its right hand side belongs to
Im(${\calH}_j^0$). Then, we multiply each line  by $\omega_k$ and add over $k$ to obtain
\begin{align*}
0 = & \rho_j^{n}-\rho_j^{n+1} + \frac{\Delta t}{\Delta x}
\left(\frac{-e^{-\lambda_{0,j-\frac 12}^1\Delta x}\rho_{j-1}^n+\rho_j^n}
{e^{-\lambda_{0,j-\frac 12}^1\Delta x}-1}
\sum_{k=1}^K \omega_k v_k \phi(v_k\pa_xS_{j-\frac 12}) \right.  \\
&\left. + \rho_j^n \sum_{k=1}^K \omega_k v_k \phi(v_k\pa_xS_{j+\frac 12})
+ \frac{\lambda_{0,j+\frac 12}^1(\rho_j^n-\rho_{j+1}^n)}{e^{-\lambda_{0,j+\frac 12}^1\Delta x}-1}\sum_{k=1}^K \omega_k v_k^2
\right),
\end{align*}
where we have used also $\forall \ell$,  
$\sum_{k=1}^K \omega_k v_k (\zeta^0 \gamma)_{k\ell} = 0$.
From Lemma \ref{lem:lambda} and assumptions (\ref{eq:restrict}), it comes
$$
\frac 13 \lambda_0^1 = 
\sum_{k=1}^K \omega_k v_k \phi(v_k \pa_xS).
$$
Denoting $E_{j-\frac 12}:=\sum_{k=1}^K \omega_k v_k \phi(v_k \pa_xS_{j-\frac 12})$,
the above scheme rewrites as
$$
\rho_j^{n+1} = \rho_j^{n} + \frac{\Delta t}{\Delta x}
\left(-E_{j-\frac 12}\frac{e^{-3 E_{j-\frac 12} \Delta x}\rho_{j-1}^n-\rho_j^n}
{e^{-3 E_{j-\frac 12} \Delta x}-1}
+  E_{j+\frac 12}\frac{e^{-3 E_{j+\frac 12}\Delta x}\rho_j^n-\rho_{j+1}^n}{e^{-3 E_{j+\frac 12}\Delta x}-1}
\right),
$$
in which we recognize the Sharfetter-Gummel scheme \eqref{SGchemo}.
\qed

\section{Diffusive limit of Vlasov-Fokker-Planck kinetic equations}\label{sec:6}

\subsection{Presentation of the continuous model}

This kinetic model reads, see \cite{NPS,PS,WO1}, in parabolic scaling,
\begin{equation}\label{VFP}
\eps \pa_tf^\eps+v\pa_xf^\eps + E \cdot \pa_v f^\eps =
\frac{1}{\eps} \pa_v\left (vf^\eps + \kappa \pa_v f^\eps\right), \qquad 0<\eps \ll 1. 
\end{equation}
The electric field is denoted $E$ and may derive from a potential $\phi$ through the relation $E=\pm \pa_x \phi$. When such a potential $\phi$ depends self-consistently on the macroscopic density of electrons (through the Poisson equation), one speaks about the Vlasov-Poisson-Fokker-Planck system. To keep the exposition simple, we shall assume that a steady electric field $E(x)$ is given. 
When $\eps \to 0$, the kinetic distribution in \eqref{VFP} relaxes to a Maxwellian, 
$$
f^\eps\to \rho^0 \frac{1}{\sqrt{2\pi\kappa}} e^{-v^2/2\kappa},
$$
where the macroscopic density $\rho^0$ solves a drift-diffusion (continuity) equation,
\begin{equation}\label{dd}
\boxed{
\pa_t \rho^0 + \kappa\pa_x \left (\frac{E}{\kappa}\rho^0 - \pa_x \rho^0 \right) = 0.}
\end{equation}
The Sharfetter-Gummel scheme associated to this system reads
\begin{align}\label{SGvfp}
\rho_j^{n+1}& = \rho_j^n + \frac{\Delta t}{\Delta x}
\big(\bar{\calJ}_{j-\frac 12}^n-\bar{\calJ}_{j+\frac 12}^n\big), \\
\bar{\calJ}_{j-1/2}^n &=
 E_{j-\frac 12} \frac{\rho_{j-1}^n - e^{-E_{j-\frac 12}\Delta x/\kappa}\rho_j^n}
{1-e^{-E_{j-\frac 12}\Delta x/\kappa}}. \nonumber
\end{align}

\subsection{Spectral decomposition of stationary solutions}

Consider the Fokker-Planck stationary problem with inflow boundaries,
\begin{equation}\label{stat2}
v\pa_x \bar f = \frac{1}{\eps} \pa_v \Big((v-\eps E)\bar f
+ \kappa \pa_v \bar f\Big), \qquad \eps,\kappa>0.
\end{equation}
A convenient ``separated variables'' ansatz reads now, (see e.g. \cite{BuTitu})
$$
\bar f(x,v)=\exp(-\lambda x-\mu v)\psi_\lambda(v),
$$
so that one recovers a standard Sturm-Liouville eigenvalue problem (see \cite{beals} and \cite[Chapter 12.3]{book}) with a discrete spectrum. The null eigenvalue $\lambda=0$ is double, its two associated non-damped modes are denoted $\Psi_{\pm 0}^\eps$ (``diffusion solutions'' in \cite{FK}) among which appears a space-homogeneous mode:
$$
\Psi_{\pm 0}^\eps(x,v) = \exp(-\frac{\mu_{\pm 0}^\eps x}{\eps}) \psi^\eps_{\pm \ell}(v),
$$
where
\begin{itemize}
\item when $E>0$,
\begin{align*}
\mu_0^\eps = 0; \quad &\psi_{0}^\eps(v) = \exp ( -\frac{(v-|\eps E|)^2}{2\kappa} ) ; 
\\
\mu_{-0}^\eps = -\frac{\eps E}{\kappa} ;\quad &\psi_{-0}^\eps(v) = \exp (-\frac{|\eps E|^2}{2\kappa}) \exp (- \frac{v^2}{2\kappa});
\end{align*}
\item when $E<0$,
\begin{align*}
\mu_{0}^\eps = -\frac{\eps E}{\kappa} ;\quad &\psi_{0}^\eps(v) = \exp (-\frac{|\eps E|^2}{2\kappa}) \exp (- \frac{v^2}{2\kappa}) ;
\\ 
\mu_{-0}^\eps = 0; \quad & \psi_{-0}^\eps(v) = \exp ( -\frac{(v+|\eps E|)^2}{2\kappa} ).
\end{align*}
\end{itemize}
Other eigenfunctions $\Psi_{\pm \ell}^\eps$, for $\ell\in\NN^*$, are explicitly given in \cite[p. 251]{book}:
\begin{eqnarray}
\Psi_{\pm \ell}^\eps(x,v) = \exp ( - \frac{\mu_{\pm \ell}^\eps x}{\eps} )\, \psi_{\pm \ell}^\eps(v) \ ,
\label{eq:Psi} \\
\psi_{\pm\ell}^\eps(v) = \exp ( -\mu_{\pm\ell}^\eps v) H_\ell(\tilde{v}_{\pm\ell}^\eps )
\exp(-(\tilde{v}_{\pm\ell}^\eps)^2)\ , \label{eq:psi}
\end{eqnarray}
where $H_\ell$ is the $\ell$th Hermite polynomial and
$$
\mu_{\pm\ell}^\eps = \frac{-\eps E\pm\sqrt{(\eps E)^2 +4\kappa \ell}}{2\kappa}, \qquad
\tilde{v}_{\pm\ell}^\eps = \frac{v-2\mu_{\pm\ell}^\eps \kappa - \eps E}{\sqrt{2 \kappa}}.
$$
Hermite's polynomials are such that deg$(H_\ell)=\ell$ and $H_\ell(-X)=(-1)^\ell H_\ell(X)$.
They are orthogonal with respect to a strongly-growing (indefinite) weight:
\begin{equation}\label{ortho-psi}
\int_\RR v\, \psi_{\pm k}^\eps(v)\psi_{\pm \ell}^\eps(v)\, \exp\left(\frac{(v-\eps E)^2}{2\kappa}\right)dv=0, \qquad \mbox{ if } k \not = \ell.
\end{equation}
A spectral decomposition follows, for smooth enough functions $\bar f(x,v)$ \cite{BP,CS},
\begin{equation}\label{vfp-ortho}
\boxed{
\bar f(x,v)=\alpha_{+0} \Psi_0^\eps(x,v) + \alpha_{-0} \Psi_{-0}^\eps (x,v)
+ \sum_{\ell \geq 1} \left(\alpha_{\ell} \Psi^\eps_{\ell}(x,v)+\alpha_{-\ell} \Psi^\eps_{-\ell}(x,v)\right).}
\end{equation}
By inserting the space-homogeneous eigenfunction associated to $k=0$, one sees from \eqref{ortho-psi} 
that none in all the set of $\Psi_{\pm \ell}^\eps$, $\ell>0$, can carry any net macroscopic flux: 
\begin{equation}\label{ortho-psi0}
\forall \ell \not=0, \quad
\int_\RR v\psi_{\pm \ell}^\eps(v)\psi_{\pm 0}^\eps(v)\exp\left(\frac{(v-\eps E)^2}{2\kappa}\right) dv=\int_\RR v\psi_{\pm \ell}^\eps(v) dv=0.
\end{equation}

For future use, we compute the limit as $\eps\to 0$ of the above expressions.
We obtain straightforwardly, as $\eps\to 0$,
\begin{subequations}
\label{eq:psi0}
\begin{gather}
\psi_{\pm 0}^0(v) = \exp\Big(-\frac{v^2}{2\kappa}\Big); \label{eq:psi00} \\
\psi_{\pm \ell}^0(v) = H_\ell\Big(\frac{v\mp 2\sqrt{\ell \kappa}}{\sqrt{2\kappa}}\Big) 
\exp\left(-\frac{v^2}{2\kappa}\pm v\sqrt{\frac{\ell}{\kappa}} - 2\ell\right),
\quad \mbox{ for } \ell\in \NN^*. \label{eq:psi0l}
\end{gather}
\end{subequations}
We deduce from the symmetry of Hermite polynomials that we have the identity 
$\psi^0_{-\ell}(-v)=(-1)^\ell \psi^0_\ell(v)$. 
These expressions do not depend on $E \in \RR$.

\subsection{Assumptions on the set of discrete velocities and weights.}

As mentioned in the introduction, and contrary to the Case's functions,
being exponential polynomials, solutions of the stationary Vlasov-Fokker-Planck equation, 
do not constitute a Chebyshev $T$-system on $(0,+\infty)$.

Indeed, if we assume that the family 
$\{\psi_0^0, \ldots, \psi_{K-1}^0\}$ is endowed with 
the Haar property, then any non-zero linear combination of these functions 
will have no more than $K-1$ roots on $(0,+\infty)$;
otherwise, denoting $v_0,\ldots,v_{K-1}$ these roots, 
det$(\psi_\ell^0(v_k))_{k,\ell}=0$  (see Proposition \ref{prop1:Haar}).
However, it is not difficult to find ad-hoc coefficients $a_0,\ldots,a_{K-1}$, for which the
function $$\RR^+ \ni v\mapsto a_0\, \psi_0^0(v) + \sum_{\ell=1}^{K-1} a_\ell\, \psi_\ell^0(v)$$
admits more than $K-1$ roots on $(0,+\infty)$.
For instance, 
\begin{itemize}
\item for $K=2$, the function (after simplification by $\exp(-v^2/2\kappa)$),
$$v\mapsto \frac{3}{2} + \frac{v-2}{\sqrt{2}} e^v$$
has two positive roots, approximately given by 0.1216 and 1.5495;
\item For $K=3$, the function $$v\mapsto
-2.75+0.2 e^v (v-2) + e^{\sqrt{2}v}\left(\frac{v^2}{2}-2\sqrt{2}v+3\right)$$
admits three positive roots, approximately given by 0.132, 0.796 and 4.192.
\end{itemize}•
The number of real roots for exponential polynomials admits the
P\'olya-Szeg\"o estimate as an upper bound, see Appendix \ref{sec:polya}.
However, exponential monomials do satisfy the Haar property on $(0,+\infty)$ as shown in Theorem \ref{th:alternant}.

Accordingly, we must prescribe some assumptions on the set of
discrete velocities. More precisely, we assume that $v_1,\ldots,v_K$
are chosen such that, 
\begin{equation}\label{hypV1}\mbox{given } K \in \NN, \qquad
\mbox{det}\Big(\psi_0^0(\calV) \quad \psi_1^0(\calV) \ \ldots \ \psi_{K-1}^0(\calV)\Big) \neq 0,
\end{equation}
and such that the family
\begin{align}
\Big\{ & \begin{pmatrix} \psi_0^0(\calV) \\ \psi_0^0(-\calV) \end{pmatrix},
\begin{pmatrix} \psi_1^0(\calV) \\ \psi_1^0(-\calV) \end{pmatrix}, \cdots,
\begin{pmatrix} \psi_{K-1}^0(\calV) \\ \psi_{K-1}^0(-\calV) \end{pmatrix},  \nonumber \\
&\begin{pmatrix} \psi_1^0(-\calV) \\ \psi_1^0(\calV) \end{pmatrix}, \cdots,
\begin{pmatrix} \psi_{K-1}^0(-\calV) \\ \psi_{K-1}^0(\calV) \end{pmatrix}
\Big\} 
\quad \mbox{ is linearly independent in } \RR^{2K}.
\label{hypV2}
\end{align}
We assume moreover that the corresponding weights $\omega_k$ 
are such that the orthogonality relation \eqref{ortho-psi0} holds true
at the discrete level, i.e.
\begin{equation}\label{eq:ortho-psi-dis}
\forall \ell=1,\ldots, K-1, \qquad 
\sum_{k=1}^K \omega_k v_k (\psi^0_{\pm \ell}(v_k)-\psi^0_{\pm \ell}(-v_k)) = 0.
\end{equation}
Thanks to the relation $\psi^0_{-\ell}(-v_k)=(-1)^\ell\psi^0_{\ell}(v_k)$, it suffices to get the above identity only for positive $\ell$.
%

Finally we mention the following Lemma which shows that actually assumptions \eqref{hypV1} and \eqref{hypV2} hold at the continuous level.
\begin{lemma}\label{free-lem}
For any $N  \in \NN$, the family $(\psi^0_0, \psi^0_1, \cdots, \psi^0_{N-1})$ given by (\ref{eq:psi}) is linearly independent on $v \in (0, +\infty)$, and also on $v\in \RR$. 
\end{lemma}
%
%
\begin{proof}
Given $N$, we have to show the linear independence on $\RR$ and on $(0,+\infty)$.
The result on $\RR$ is a direct consequence of the orthogonality relation \eqref{ortho-psi}.
Thus we are left to prove the result on $(0,+\infty)$. After expanding the Gaussian term in (\ref{eq:psi}) and normalizing coefficients, any linear combination with $\eps=0$ rewrites,
$$\forall v \in (0,+\infty), \qquad 
\sum_{i=0}^{N-1} \lambda_i \exp(\mu_i \, v)H_i(\tilde v_i)=0.
$$
Since $H_{N-1} \not =0$ for $v$ big enough, it suffices to let $v \to + \infty$ in
$$
\lambda_{N-1} + \sum_{i=0}^{N-2} \lambda_i \exp((\mu_i-\mu_{N-1}) \, v)\frac{H_i(\tilde v_i)}{H_{N-1}(\tilde v_{N-1})} =0,
$$
in order to get $\lambda_{N-1}=0$ because $\mu_i-\mu_{N-1} \leq C <0$. Successive coefficients vanish for the same reason. Hence for any $N$, the family $(\psi_i^0)_{i<N}$ is linearly independent on $\RR^+_*$. 
\qed\end{proof}

\subsection{Corresponding scattering $S$-matrix}

The stationary problem with incoming boundary condition reads
\begin{eqnarray}
&&\label{eqWB1} v\pa_x \barf = \frac{1}{\eps} \pa_v \big((v-\eps E_{j-\frac 12})\barf 
+ \kappa \pa_v \barf\big), \qquad \mbox{ on } (0,\Delta x),\\[2mm]
&&\label{eqWB2} \barf(0,v) = f_{j-1}^{n}(v), \quad \barf(\Delta x,-v) = f_{j}^{n}(-v).
\end{eqnarray}
Based on (\ref{vfp-ortho}) and \eqref{eq:Psi}, we seek spectral approximations obtained by truncating \eqref{vfp-ortho} to the first $2K$ modes,
\begin{equation}\label{eq:fa}
\boxed{
\barf(x,v)= \alpha_{+0} \psi_{0}^\eps(v) e^{-\mu_{0}^\eps x/\eps} + \alpha_{-0} \psi_{-0}^\eps(v)e^{-\mu_{-0}^\eps x/\eps} 
+ \sum_{\ell =1}^{K-1} \alpha_{\pm \ell} \psi_{\pm \ell}^\eps(v)e^{-\mu_{\pm \ell}^\eps x/\eps}.}
\end{equation}
Coefficients $\alpha_{\pm \ell}$ in this full-range expansion follow from incoming boundary conditions \eqref{eqWB2}, according to a linear system: for $k=1,\ldots,K$,
\begin{align*}
\barf(0,v_k)=f_{j-1}^{n}(v_k) =& \alpha_0 \psi_0^\eps(v_k) + \alpha_{-0} \psi_{-0}^\eps(v_j) + \sum_{\ell =1}^{K-1} \alpha_{\pm\ell} \psi_{\pm\ell}^\eps(v_k), \\
\barf(\Delta x,-v_k)=f_{i+1}^{n}(-v_k) =& \alpha_0 \psi_0^\eps(- v_k) e^{-\frac{\mu_{0}}{\eps}\Delta x}  + \alpha_{-0} \psi_{-0}^\eps(-v_k) e^{-\frac{\mu_{-0}^\eps}{\eps} \Delta x} 
\\&\qquad \qquad \qquad + \sum_{\ell=1}^{K-1} \alpha_{\pm\ell} \psi_{\pm\ell}^\eps(-v_k)e^{-\frac{\mu_{\pm \ell}^\eps}{\eps}\Delta x}.
\end{align*}
%
The $K\times (K-1)$ matrix of eigenvectors associated to nonzero eigenvalues is
$$
\psi_\pm^\eps(\calV)=\begin{pmatrix} \psi_{\pm 1}^\eps(\calV) & \ldots  &  \psi_{\pm K-1}^\eps(\calV) \end{pmatrix},
$$
so that,
$$
\begin{pmatrix} \alpha_+ \\ \alpha_- \end{pmatrix} = (\calM^\eps)^{-1} \begin{pmatrix} f_{j-1}^{n}(\calV) \\ f_{j}^{n}(-\calV) \end{pmatrix},
$$
where $\calM^\eps$ is the $2 K\times 2 K$ matrix defined by
\begin{equation*}
\calM^\eps=\begin{pmatrix} 
\psi_+^\eps(\calV) & \psi_{0}^\eps(\calV) & \psi_-^\eps(\calV) & \psi_{-0}^\eps(\calV) \\ 
\psi_+^\eps(-\calV) e^{-\mu_{+}^\eps\Delta x/\eps} & \psi_0^\eps(-\calV) e^{-\mu_0^\eps \Delta x/\eps} & \psi_-^\eps(-\calV) e^{-\mu_{-}^\eps\Delta x/\eps} & \psi_{-0}^\eps(-\calV) e^{-\mu_{-0}^\eps\Delta x/\eps}
\end{pmatrix}.
\end{equation*}
Outgoing values of $\barf$ follow thanks to \eqref{eq:fa},
\begin{eqnarray*}
& \barf(\Delta x,\calV) = &\alpha_0 \psi_0^\eps(\calV) e^{-\mu_{0}^\eps\Delta x/\eps} + \alpha_{-0} \psi_{-0}^\eps(\calV)e^{-\mu_{-0}^\eps\Delta x/\eps} \\
&&+ \sum_{\ell=1}^{K-1} \alpha_{\pm \ell} \psi_{\pm\ell}^\eps(\calV)e^{-\mu_{\pm \ell}^\eps \Delta x/\eps}, \\
& \barf(0,-\calV) = &\alpha_0 \psi_0^\eps(-\calV) + \alpha_{-0} \psi_{-0}^\eps(-\calV) + \sum_{\ell=1}^{K-1} \alpha_{\pm\ell} \psi_{\pm\ell}^\eps(-\calV).
\end{eqnarray*}
Then, denoting 
\begin{equation*}
\widetilde{\calM^\eps}=\begin{pmatrix} \psi_+^\eps(\calV)e^{-\mu_{+}^\eps\Delta x/\eps} & \psi_0^\eps(\calV) e^{-\mu_0^\eps\Delta x/\eps} & \psi_-^\eps(\calV)e^{-\mu_{-}^\eps\Delta x/\eps} & \psi_{-0}^\eps(\calV)e^{-\mu_{-0}^\eps\Delta x/\eps}  \\ 
\psi_+^\eps(-\calV) &  \psi_0^\eps(-\calV) & \psi_-^\eps(-\calV) & \psi_{-0}^\eps(-\calV)
\end{pmatrix},
\end{equation*}
we get
$$
\begin{pmatrix} \ftilde(\Delta x,\calV) \\ \ftilde(0,-\calV) \end{pmatrix} = \widetilde{\calM^\eps} (\calM^\eps)^{-1} \begin{pmatrix} f_{j-1}^{n}(\calV) \\ f_{j}^{n}(-\calV) \end{pmatrix}.
$$
Accordingly, the scattering matrix is defined by, 
$$
\calS^{\eps}_{j-1/2}=\widetilde{\calM^\eps} (\calM^\eps)^{-1}.\qquad \mbox{ (dependent on $j-\frac 1 2$)}
$$

\subsection{Decomposition of the scattering matrix}

We perform the computation in the case $E>0$. Computations in the case $E<0$ are similar.
We start from the expression, 
$$
\calS^\eps = \widetilde{\calM^\eps} (\calM^\eps)^{-1} = \widetilde{N^\eps} (N^\eps)^{-1},
$$
where (expressions of $\mu_{\pm 0}^\eps$ are used),
$$
N^\eps=\begin{pmatrix} 
\psi_+^\eps(\calV) & \psi_{0}^\eps(\calV) & \psi_-^\eps(\calV) e^{\mu_{-}^\eps\Delta x/\eps} & \psi_{-0}^\eps(\calV)  e^{- E\Delta x/\kappa} \\ 
\psi_+^\eps(-\calV) e^{-\mu_{+}^\eps\Delta x/\eps} & \psi_0^\eps(-\calV) & \psi_-^\eps(-\calV) & \psi_{-0}^\eps(-\calV)
\end{pmatrix},
$$
$$
\widetilde{N^\eps}=\begin{pmatrix} \psi_+^\eps(\calV)e^{-\mu_{+}^\eps\Delta x/\eps} & \psi_0^\eps(\calV)  & \psi_-^\eps(\calV) & \psi_{-0}^\eps(\calV)  \\ 
\psi_+^\eps(-\calV) &  \psi_0^\eps(-\calV) & \psi_-^\eps(-\calV)e^{\mu_{-}^\eps\Delta x/\eps} & \psi_{-0}^\eps(-\calV)e^{-E\Delta x/\kappa}
\end{pmatrix}.
$$
Moreover, the product $\widetilde{N^\eps} (N^\eps)^{-1}$ is invariant if, in both matrices $N^\eps$ and $\widetilde{N^\eps}$, we subtract to the last column $e^{-E\Delta x/\kappa}$ times the $K$th column, i.e. 
\begin{equation}\label{scat:vfp}
\calS^\eps = \widetilde{\calN^\eps} (\calN^\eps)^{-1},
\end{equation}•
with
\begin{equation*}
\calN^\eps=\begin{pmatrix} 
\psi_+^\eps(\calV) & \psi_{0}^\eps(\calV) & \psi_-^\eps(\calV) e^{\mu_{-}^\eps\Delta x/\eps} & (\psi_{-0}^\eps(\calV)-\psi_{0}^\eps(\calV))  e^{- E\Delta x/\kappa} \\ 
\psi_+^\eps(-\calV) e^{-\mu_{+}^\eps\Delta x/\eps} & \psi_0^\eps(-\calV) & \psi_-^\eps(-\calV) & \psi_{-0}^\eps(-\calV)-\psi_{0}^\eps(-\calV) e^{-E \Delta x/\kappa}
\end{pmatrix},
\end{equation*}
\begin{equation*}
\widetilde{\calN^\eps}=\begin{pmatrix} \psi_+^\eps(\calV)e^{-\mu_{+}^\eps\Delta x/\eps} & \psi_0^\eps(\calV)  & \psi_-^\eps(\calV) & \psi_{-0}^\eps(\calV) - \psi_0^\eps(\calV) e^{-E \Delta x/\kappa} \\ 
\psi_+^\eps(-\calV) &  \psi_0^\eps(-\calV) & \psi_-^\eps(-\calV)e^{\mu_{-}^\eps\Delta x/\eps} & (\psi_{-0}^\eps(-\calV)-\psi_0^\eps(-\calV))e^{-E\Delta x/\kappa}
\end{pmatrix}.
\end{equation*}
Noticing that when $\eps\to 0$, $\psi_{\pm 0}^\eps (\calV) \to \exp(-\frac{\calV^2}{2\kappa})$, we may pass to the limit $\eps \to 0$ in the latter matrices and get $\calS^\eps \to \calS^0 = \calN^0 (\widetilde{\calN^0})^{-1}$ with
\begin{equation}\label{eq:N0}
\calN^0=\begin{pmatrix} 
\psi_+^0(\calV) & \exp(-\frac{\calV^2}{2\kappa}) & \mathbf{0}_{K} \\ 
\mathbf{0}_{K\times K-1} & \exp(-\frac{\calV^2}{2\kappa}) &
\Big(\psi_-^0(-\calV) \quad (1- e^{-E \Delta x/\kappa})\exp(-\frac{\calV^2}{2\kappa})\Big)
\end{pmatrix},
\end{equation}
\begin{equation}\label{eq:N0tilde}
\widetilde{\calN^0}=\begin{pmatrix} 
\mathbf{0}_{K\times K-1} & \exp(-\frac{\calV^2}{2\kappa})  & 
\Big(\psi_-^0(\calV) \quad (1 - e^{-E \Delta x/\kappa})\exp(-\frac{\calV^2}{2\kappa})\Big) \\ 
\psi_+^0(-\calV) &  \exp(-\frac{\calV^2}{2\kappa}) & \mathbf{0}_K
\end{pmatrix}.
\end{equation}
\begin{lemma}\label{lem:scatdev2}
Under assumption \eqref{hypV1} on the set of discrete velocities,
the $S$-matrix \eqref{scat:vfp} admits the asymptotic expansion in $\eps$,
\begin{equation}\label{eq:scat-vfp}
\fbox{$\calS^\eps = \begin{pmatrix} 
\mathbf{0}_K & \mathbf{I}_K - \zeta_+^0 \gamma_+  \\
\mathbf{I}_K - \zeta_+^0 \gamma_+ & \mathbf{0}_K
\end{pmatrix}
+ \eps B^\eps \,,\quad 
B^\eps = \frac{1}{\eps} (\widetilde{\calN^\eps} (\calN^\eps)^{-1} - \widetilde{\calN^0} (\calN^0)^{-1})$,}
\end{equation}
where the matrices $\gamma_+ \in \calM_{K-1\times K}(\RR)$ and
$\zeta_+^0\in \calM_{K\times K-1}(\RR)$ satisfy,
\begin{align}
\label{def:gam+}
&\gamma_+ = \begin{pmatrix} \gamma_1^\top \\ \vdots \\ \gamma_{K-1}^\top \end{pmatrix} \mbox{ where } \, \gamma_\ell \in \RR^K, \quad \gamma_\ell^\top \psi_k^0(\calV) = \delta_{k\ell}, \quad \gamma_\ell^\top \exp(-\frac{\calV^2}{2\kappa}) = 0; \\
\label{def:zeta+}
&\zeta_\pm^\eps = \begin{pmatrix} \zeta_{\pm 1}^\eps & \dots & \zeta_{\pm (K-1)}^\eps \end{pmatrix}, \qquad 
\zeta_{\pm \ell}^\eps = \psi_{\pm \ell}^\eps(\calV)-\psi_{\pm \ell}^\eps(-\calV)\in \RR^K. \quad
\end{align}
\end{lemma}
\begin{remark}
Existence of $\gamma_+$ is guaranteed by assumption \eqref{hypV1}.
Notice that since the limits $\psi^0_{\pm \ell}$ in \eqref{eq:psi0} do not depend on $E$, the first term in the decomposition of the scattering matrix is independent on $E$, hence on $j$.
\end{remark}
\begin{proof}
\begin{enumerate}
\item Let vector $\beta\in \RR^K$ be such that
\begin{equation}\label{def:betavfp}
\beta^\top \psi_\ell^0(\calV) = 0; \quad \beta^\top \exp(-\frac{\calV^2}{2\kappa}) = 1.
\end{equation}
Moreover, since 
\begin{equation}\label{inv1}
\gamma_+ = \begin{pmatrix} \gamma_1^\top \\ \vdots \\ \gamma_{K-1}^\top \end{pmatrix}, 
\mbox{ it comes that } 
\begin{pmatrix} \gamma_+ \\ \beta^\top \end{pmatrix} = \begin{pmatrix} 
\psi_+^0(\calV) &\ \exp(-\frac{\calV^2}{2\kappa}) \end{pmatrix}^{-1},
\end{equation}
where the latter matrix is invertible thanks to assumption \eqref{hypV1}.
It brings also that,
\begin{equation}\label{betapsi}
\sum_{\ell=1}^{K-1} \psi_\ell^0(\calV)\gamma_\ell^\top + \exp(-\frac{\calV^2}{2\kappa})\beta^\top=I_K.
\end{equation}
\item Similarly, as $\psi_{-\ell}^0 (\calV) = (-1)^\ell \psi_{\ell}^0 (-\calV)$, then denoting 
$$\gamma_- = \begin{pmatrix} -\gamma_1^\top \\ \gamma_2^\top \\ \vdots \\ (-1)^{K-1}\gamma_{K-1}^\top \end{pmatrix},
$$
it comes that,
\begin{equation}\label{inv2}
\begin{pmatrix} \gamma_- \\ \frac{1}{1-e^{-E\Delta x/\kappa}}\beta^\top \end{pmatrix} = \begin{pmatrix} 
\psi_-^0(-\calV) &\ (1-e^{-E\Delta x/\kappa})\exp(-\frac{\calV^2}{2\kappa}) \end{pmatrix}^{-1}.
\end{equation}
\item We deduce from both \eqref{inv1} and \eqref{inv2} that,
$$
(\calN^0)^{-1} = 
\begin{pmatrix}
\begin{pmatrix} \gamma_+ \\ \beta^\top \end{pmatrix} & \mathbf{0}_K \\
\begin{pmatrix} \mathbf{0}_{K-1\times K} \\ -\frac{1}{1-e^{-E\Delta x/\kappa}}\beta^\top \end{pmatrix} & 
\begin{pmatrix} \gamma_- \\ \frac{1}{1-e^{-E\Delta x/\kappa}} \beta^\top \end{pmatrix} \\
\end{pmatrix},
$$
and so, with (\ref{eq:N0})--(\ref{eq:N0tilde}),
$$
\calS^\eps \underset{\eps\to 0}{\longrightarrow} \calS^0 := \widetilde{\calN^0}(\calN^0)^{-1},
$$
has the form \eqref{eq:scat-vfp}, with 
\begin{align*}
& \widetilde{\calN^0}(\calN^0)^{-1} = \\
& \begin{pmatrix}
\mathbf{0}_K & \displaystyle \sum_{\ell=1}^{K-1} \psi_\ell^0(-\calV)\gamma_\ell^\top + \exp(-\frac{\calV^2}{2\kappa})\beta^\top \\
\displaystyle \sum_{\ell=1}^{K-1} \psi_\ell^0(-\calV)\gamma_\ell^\top + \exp(-\frac{\calV^2}{2\kappa})\beta^\top & \mathbf{0}_K
\end{pmatrix}.
\end{align*}
\end{enumerate}
Using \eqref{betapsi} allows to complete the proof.
\qed
\end{proof}

\begin{lemma}
For $E>0$, 
\begin{equation}\label{Bzero}
B^0:= \lim_{\eps\to 0} B^\eps = \begin{pmatrix}
B^{10} & B^{20} \\ B^{30} & B^{40}
\end{pmatrix},
\end{equation}
where 
\begin{align*}
B^{10} = &
\frac{1}{1-e^{-\frac{E\Delta x}{\kappa}}} \frac {2 E \calV}{\kappa} \exp(-\frac{\calV^2}{2\kappa})\beta^\top
- \frac{1}{1-e^{-\frac{E\Delta x}{\kappa}}} \zeta^0_+\gamma_+ \left(\frac {E \calV}{\kappa} - \frac{E^2}{\kappa}\right) \exp(-\frac{\calV^2}{2\kappa}) \beta^\top ;  \\
B^{20} = & \left(\frac{d\zeta_-^\eps}{d\eps}_{|\eps=0} + \zeta_+^0 \gamma_+ \frac{d\psi_-^\eps(-\calV)}{d\eps}_{|\eps=0}\right)\gamma_- 
- \frac{e^{-\frac{E\Delta x}{\kappa}}}{1-e^{-\frac{E\Delta x}{\kappa}}}\frac {2 E \calV}{\kappa} \exp(-\frac{\calV^2}{2\kappa})\beta^\top  \\
& \quad - \frac{1}{1-e^{-\frac{E\Delta x}{\kappa}}} \zeta_+^0\gamma_+ 
\Big(\frac {E^2}{\kappa} + \frac {2 E \calV}{\kappa} e^{-\frac{E\Delta x}{\kappa}}\Big) \exp(-\frac{\calV^2}{2\kappa}) \beta^\top ;
  \\
B^{30} = & \left( \zeta_+^0 \gamma_+ \frac{d\psi_-^\eps(\calV)}{d\eps}_{|\eps=0} - \frac{d\zeta_+^\eps}{d\eps}_{|\eps=0} \right)\gamma_+
- \frac{1}{1-e^{-\frac{E\Delta x}{\kappa}}} \frac {2 E \calV}{\kappa} \exp(-\frac{\calV^2}{2\kappa}) \beta^\top  \\
&\quad + \frac{1}{1-e^{-\frac{E\Delta x}{\kappa}}} \zeta_+^0\gamma_+ \left(\frac{E\calV}{\kappa}+\frac{E^2}{\kappa}\right)\exp(-\frac{\calV^2}{2\kappa}) \beta^\top ; \\
B^{40} = & \frac{e^{-\frac{E\Delta x}{\kappa}}}{1-e^{-\frac{E\Delta x}{\kappa}}}
\left(\frac{2E\calV}{\kappa}-\zeta_+^0\gamma_+\left(\frac{E^2}{\kappa}+\frac{E\calV}{\kappa}\right)\right)\exp(-\frac{\calV^2}{2\kappa}) \beta^\top.  
\end{align*}
\end{lemma}
\begin{proof}
Rewriting $B^\eps = \frac{1}{\eps}(\widetilde{\calN^\eps}-\widetilde{\calN^0}(\calN^0)^{-1}\calN^\eps)(\calN^\eps)^{-1}$, it comes,
$$
B^\eps \calN^\eps =\frac{1}{\eps}(\widetilde{\calN^\eps}-\widetilde{\calN^0}(\calN^0)^{-1}\calN^\eps) = \frac 1\eps
\begin{pmatrix}
A^{1\eps} & A^{2\eps} \\ A^{3\eps} & A^{4\eps}
\end{pmatrix},
$$
with
\begin{align*}
A^{1\eps} = & \begin{pmatrix} 
(\zeta^\eps + \zeta^0_+\gamma_+ \psi_+^\eps(-\calV)) e^{-\frac{\mu_{+}^\eps\Delta x}{\eps}}
& \quad \zeta_0^\eps + \zeta^0_+\gamma_+\psi_0^\eps(-\calV) 
\end{pmatrix};  \\
A^{2\eps} = & \begin{pmatrix} 
\zeta_-^\eps + \zeta_+^0 \gamma_+ \psi_-^\eps(-\calV)
& \quad \zeta_{-0}^\eps + \zeta^0_+ \gamma_+ \psi_{-0}^\eps(-\calV)
  - (\zeta_0^\eps + \zeta_+^0\gamma_+ \psi_{0}^\eps(-\calV))  e^{-\frac{E \Delta x}{\kappa}}
\end{pmatrix};  \\ 
A^{3\eps} = & \begin{pmatrix} 
\zeta_+^0\gamma_+ \psi_+^\eps(\calV) -\zeta_+^\eps
& \quad \zeta_+^0\gamma_+ \psi_{0}^\eps(\calV) -\zeta_0^\eps
\end{pmatrix}; \\
A^{4\eps} = & \begin{pmatrix} 
(\zeta_+^0\gamma_+ \psi_-^\eps(\calV) -\zeta_-^\eps) e^{\frac{\mu_{-}^\eps\Delta x}{\eps}}
& \quad (\zeta_0^\eps-\zeta_{-0}^\eps + \zeta_+^0\gamma_+ (\psi_{-0}^\eps(\calV)-\psi_{0}^\eps(\calV))) e^{-\frac{E\Delta x}{\kappa}}
\end{pmatrix}.
\end{align*}
$$\boxed{
\mbox{When } E>0,\quad \zeta_{-0}^\eps = 0, \mbox{ whereas if } E<0, \quad \zeta_0^\eps = 0. }
$$
%
Accordingly, for $E>0$,
\begin{align*}
&\lim_{\eps\to 0} \frac{1}{\eps}\left(\psi_0^\eps(\calV) - \exp(-\frac{\calV^2}{2\kappa})\right) = \frac{d \psi_0^\eps(\calV)}{d\eps}_{|\eps=0} = \frac {E \calV}{\kappa} \exp(-\frac{\calV^2}{2\kappa});  \\
&\lim_{\eps\to 0} \frac{1}{\eps}\left(\psi_{-0}^\eps(\calV) - \exp(-\frac{\calV^2}{2\kappa})\right) = -\frac {E^2}{\kappa} \exp(-\frac{\calV^2}{2\kappa}); \\
&\lim_{\eps\to 0} \frac{1}{\eps} \zeta_0^\eps = \frac {2 E \calV}{\kappa} \exp(-\frac{\calV^2}{2\kappa}),
\end{align*}
and the expression (\ref{Bzero}) follows by doing $B^0= \left(\lim_{\eps\to 0} \frac{1}{\eps} A^\eps\right)(\calN^0)^{-1}$.
\qed\end{proof}

\subsection{Uniform accuracy with respect to $\eps$ of the final scheme}

We deduce the final scheme from \eqref{schemeAPWB}, which reads like \eqref{scheme1:chemo},
\begin{align}
\begin{pmatrix} f_j^{n+1}(\calV) \\ f_{j-1}^{n+1}(-\calV) \end{pmatrix} 
&+\frac{\Delta t}{\eps\Delta x}\VV 
\begin{pmatrix} f_j^{n+1}(\calV)-(\mathbf{I}_K-\zeta^0_{+}\gamma_{+}) f_j^{n+1}(-\calV) \\ 
f_{j-1}^{n+1}(-\calV)-(\mathbf{I}_K-\zeta^0_{+}\gamma_{+}) f_{j-1}^{n+1}(\calV) \end{pmatrix}   \nonumber \\
&=
\begin{pmatrix} f_j^{n}(\calV) \\ f_{j-1}^{n}(-\calV) \end{pmatrix}
+\frac{\Delta t}{\Delta x} \VV B^\eps_{j-\frac 12}
\begin{pmatrix} f_{j-1}^{n}(\calV) \\ f_j^{n}(-\calV) \end{pmatrix}.
\label{scheme1:vfp}
\end{align}
Denoting again $B^\eps = \begin{pmatrix} B^{1\eps} & B^{2\eps} \\ B^{3\eps} & B^{4\eps} \end{pmatrix}$,
the scheme \eqref{scheme1:vfp} rewrites as
\begin{equation}\label{scheme2:vfp}
\frac{1}{\eps}\calH^\eps 
\begin{pmatrix} f_j^{n+1}(\calV) \\ f_j^{n+1}(-\calV) \end{pmatrix} =
\begin{pmatrix} f_j^{n}(\calV)  \\ f_{j}^n(-\calV)) \end{pmatrix}
+ \frac{\Delta t}{\Delta x} \VV
\begin{pmatrix} B^{1\eps}_{j-\frac 12} f_{j-1}^n(\calV)+B^{2\eps}_{j-\frac 12} f_j^n(-\calV) \\
B^{3\eps}_{j+\frac 12} f_{j}^n(\calV)+B^{4\eps}_{j+\frac 12} f_{j+1}^n(-\calV) \end{pmatrix},
\end{equation}
where
$$
\calH^\eps = \eps \mathbf{I}_{2K} + \frac{\Delta t}{\Delta x}\VV
\begin{pmatrix} \mathbf{I}_K & \zeta^0_{+}\gamma_{+}-\mathbf{I}_K \\ 
\zeta^0_{+}\gamma_{+}-\mathbf{I}_K & \mathbf{I}_K \end{pmatrix}.
$$
After inversion of the matrix $\calH^\eps$,
by construction, the scheme (\ref{scheme1:vfp}) satisfies the well-balanced property and is asymptotic preserving. We also define:
\begin{equation}\label{def:sigma}
\sigma_0 = \sum_{k=1}^K \omega_k e^{-v_k^2/2\kappa}, \qquad
\sigma_2 = \sum_{k=1}^K \omega_k v_k^2\, e^{-v_k^2/2\kappa}.
\end{equation}
We can now state our main result for the Vlasov-Fokker-Planck equation:
\begin{theorem}\label{th:VFP}
The scheme \eqref{scheme2:vfp} is a well-balanced approximation for the Vlasov-Fokker-Planck 
system \eqref{VFP}.
It is uniformly accurate (AP) with respect to $\eps$. 
More precisely, if we assume that \eqref{hypV1}, \eqref{hypV2} and \eqref{eq:ortho-psi-dis} hold,
and that $\sigma_2=\kappa \sigma_0$ in \eqref{def:sigma}.
Then, as $\eps\to 0$, 
the macroscopic density $\rho_j^n:=\sum_{k=-K}^K \omega_k f_j^n(v_k)$ satisfies the 
Sharfetter-Gummel discretization \eqref{SGvfp}.
\end{theorem}
\begin{remark}
By integration by parts, we notice that, based on (\ref{def:sigma}),
$$
\int_{0}^{+\infty} v^2 e^{-v^2/(2\kappa)}\,dv = \kappa \int_0^{+\infty} e^{-v^2/(2\kappa)}\,dv.
$$
Therefore, assuming $\sigma_2 = \kappa \sigma_0$ in Theorem \ref{th:VFP} boils down to assume that the latter equality is also true at the discrete level.
\end{remark}
\begin{proof}
As above, we have $\calH^\eps = \calH^0 + \eps \mathbf{I}_{2K},$ where
$$
\calH^0 := \frac{\Delta t}{\Delta x} \VV
\begin{pmatrix} \mathbf{I}_K & \zeta^0_{+}\gamma _{+}- \mathbf{I}_K \\ 
\zeta^0_{+}\gamma_{+} - \mathbf{I}_K & \mathbf{I}_K \end{pmatrix}.
$$
Assuming that $f$ admits a Hilbert expansion $f=f^0 + \eps f^1 + o(\eps)$,
we get by identifying the terms in power of $\eps$ in \eqref{scheme2:vfp},
\begin{equation}\label{eqf0vfp}
\calH^0
\begin{pmatrix} \{f^0\}_j^{n+1}(\calV) \\ \{f^0\}_j^{n+1}(-\calV) \end{pmatrix} = 0,
\end{equation}
and
\begin{align}
\calH^0
\begin{pmatrix} \{f^1\}_j^{n+1}(\calV) \\ \{f^1\}_j^{n+1}(-\calV) \end{pmatrix}
= & \begin{pmatrix} \{f^0\}_j^{n}(\calV)-\{f^0\}_j^{n+1}(\calV) \\ \{f^0\}_j^{n}(-\calV)-\{f^0\}_j^{n+1}(-\calV) \end{pmatrix}
  \label{eqf1vfp}  \\
&+ \frac{\Delta t}{\Delta x} \VV 
\begin{pmatrix} B^{10}_{j-\frac 12} \{f^0\}_{j-1}^n(\calV)+B_{j-\frac 12}^{20} \{f^0\}_j^n(-\calV) \\[1mm]
 B_{j+\frac 12}^{30} \{f^0\}_{j}^n(\calV)+B_{j+\frac 12}^{40} \{f^0\}_{j+1}^n(-\calV) \end{pmatrix}. \nonumber
\end{align}
Under assumptions \eqref{hypV1}, \eqref{hypV2}, and \eqref{eq:ortho-psi-dis}, 
we may apply Lemma \ref{Hzero-vfp} in Appendix. Hence
\begin{itemize}
\item 
$\mbox{Ker}(\calH^0)=\mbox{span}\left(\exp(-\frac{\calV^2}{2\kappa})\right)=\mbox{span }\Big(\psi^0_0(\calV)\Big)$,
\item $\mbox{Im}(\calH^0) = \Big\{Z= (Z_1\ Z_2)^\top,\ Z_i\in \RR^{K} \mbox{ such that } 
\sum_{k=1}^K \omega_k ({Z_1}_k+{Z_2}_{k}) = 0\Big\}$.
\end{itemize}

Then, equation \eqref{eqf0vfp} implies that $f^0$ is an element of Ker($\calH^0$):
$$
\{f^0\}_j^{n+1}(\pm\calV)=\frac{\rho_j^{n+1}}{2\sigma_0} \exp(-\frac{\calV^2}{2\kappa}).
$$
Injecting this expression into \eqref{eqf1vfp}, we deduce,
\begin{align}
\label{eqf1vfp-b}
\calH^0
\begin{pmatrix} \{f^1\}_j^{n+1}(\calV) \\ \{f^1\}_j^{n+1}(-\calV) \end{pmatrix}
&= \frac{1}{2\sigma_0} \begin{pmatrix}
(\rho_j^{n} - \rho_j^{n+1}) \exp(-\frac{\calV^2}{2\kappa}) \\[1mm]
(\rho_j^{n} - \rho_j^{n+1}) \exp(-\frac{\calV^2}{2\kappa}) 
\end{pmatrix}
\\ & \!\!\!\!\! + \frac{\Delta t}{2\sigma_0\Delta x} \VV 
\begin{pmatrix} 
B_{j-\frac 12}^{10} \exp(-\frac{\calV^2}{2\kappa}) \rho_{j-1}^n +B^{20}_{j-\frac 12} \exp(-\frac{\calV^2}{2\kappa}) \rho_j^n \\[1mm]
 B^{30}_{j+\frac 12} \exp(-\frac{\calV^2}{2\kappa}) \rho_{j}^n+B^{40}_{j+\frac 12} \exp(-\frac{\calV^2}{2\kappa}) \rho_{j+1}^n 
\end{pmatrix}.\nonumber
\end{align}
This equation admits a solution iff the right hand side belongs to Im$(\calH^0)$.
Applying again Lemma \ref{Hzero-vfp}, we deduce the solvability condition:
\begin{align*}
  \rho_j^{n}-\rho_j^{n+1} &+ \frac{\Delta t}{2\sigma_0\Delta x}
\sum_{k=1}^K \omega_k v_k\left( (B_{j-\frac 12}^{10} \exp(-\frac{\calV^2}{2\kappa}))^{}_k \rho_{j-1}^n + (B_{j-\frac 12}^{20} \exp(-\frac{\calV^2}{2\kappa}))^{}_k\rho_j^n \right.\\
  & \left. + (B_{j+\frac 12}^{30} \exp(-\frac{\calV^2}{2\kappa}))^{}_k \rho_j^n
+ (B_{j+\frac 12}^{40} \exp(-\frac{\calV^2}{2\kappa}))^{}_k \rho_{j+1}^n
\right)=0.
\end{align*}
Using (see \eqref{def:gam+} and \eqref{def:betavfp})
$$\gamma_\pm \exp(-\frac{\calV^2}{2\kappa}) = \mathbf{0}_{\RR^{K-1}}, \qquad \beta^\top \mathbf{1}_{\RR^K} = 1,
$$ 
we deduce from the expression (\ref{Bzero}) of $B^0$,
\begin{align*}
& \sum_{k=1}^K \omega_k v_k (B_{j-\frac 12}^{10} \exp(-\frac{\calV^2}{2\kappa}))_k  =
\frac{1}{1-e^{-E_{j-\frac 12}\Delta x/\kappa}} \frac {2 E_{j-\frac 12}}{\kappa} \sum_{k=1}^K \omega_k v_k^2 \exp(-\frac{v_k^2}{2\kappa});  \\
&\sum_{k=1}^K \omega_k v_k (B_{j+\frac 12}^{20} \exp(-\frac{\calV^2}{2\kappa}))_k = 
- \frac{e^{-E_{j-\frac 12}\Delta x/\kappa}}{1-e^{-E_{j-\frac 12}\Delta x/\kappa}}\frac {2 E_{j-\frac 12}}{\kappa} \sum_{k=1}^K \omega_k v_k^2 \exp(-\frac{v_k^2}{2\kappa}); \\
&\sum_{k=1}^K \omega_k v_k (B_{j+\frac 12}^{30} \exp(-\frac{\calV^2}{2\kappa}))_k = 
- \frac{1}{1-e^{-E_{j+\frac 12}\Delta x/\kappa}} \frac {2 E_{j+\frac 12}}{\kappa} \sum_{k=1}^K \omega_k v_k^2 \exp(-\frac{v_k^2}{2\kappa});  \\
&\sum_{k=1}^K \omega_k v_k (B_{j+\frac 12}^{40} \exp(-\frac{\calV^2}{2\kappa}))_k = 
\frac{e^{-E_{j+\frac 12}\Delta x/\kappa}}{1-e^{-E_{j+\frac 12}\Delta x/\kappa}}
\frac{2E_{j-\frac 12}}{\kappa} \sum_{k=1}^K \omega_k v_k^2 \exp(-\frac{v_k^2}{2\kappa}).
\end{align*}
Thus, we may rewrite the limiting scheme as
\begin{align*}
0 &= \rho_j^{n}-\rho_j^{n+1} \\
&+ \frac{\sigma_2 \Delta t}{\sigma_0 \Delta x}
\left(\frac{E_{j-\frac 12}}{\kappa}\frac{\rho_{j-1}^n-e^{-E_{j-\frac 12} \Delta x/\kappa}\rho_j^n}
{1-e^{-E_{j-\frac 12} \Delta x/\kappa}}
- \frac{E_{j+\frac 12}}{\kappa}\frac{\rho_j^n-e^{-E_{j+\frac 12}\Delta x/\kappa}\rho_{j+1}^n}{1-e^{-E_{j+\frac 12}\Delta x/\kappa}}
\right).
\end{align*}
By assumption $\sigma_2 = \kappa\sigma_0$, we recognize the Sharfetter-Gummel scheme \eqref{SGvfp}.
\qed\end{proof}

\section{Conclusion and outlook}

In this paper, we proved (at the price of heavy technicalities) that most the discretizations advocated in \cite[Part II]{book}, for $1+1$ kinetic models endowed with a non-trivial diffusive limit, converge toward Il'in/Scharfetter-Gummel's ``exponential-fitting'' scheme. Such a property implies that IMEX discretizations (\ref{scheme2:chemo}) and (\ref{scheme2:vfp}), which rely on $S$-matrices, should be ``uniformly accurate'' (in the sense of, {\it e.g.}, \cite{berger,gartland,roos-book}) with respect to the Peclet number,
$$
\mbox{Pe}=|E_{j-\frac 1 2}| \mbox{ for (\ref{SGchemo})}, \qquad 
\mbox{Pe}=\left|\frac{E_{j-\frac 1 2}}{\kappa}\right|  \mbox{ for (\ref{SGvfp})}.
$$
In order to ensure strong stability properties for the overall discretization, the $S$-matrices should be endowed with left/right-stochastic properties, as recalled in Lemma \ref{lem:CGT}: in particular, the requirement (\ref{eq:massc}) is important for mass-preservation. This is closely related to ``{\it matrix balancing} techniques'' and Sinkhorn's algorithm, \cite{HRS,sink}. Especially, results in \cite{HRS} ensure that one can often adjust a slightly noisy $S$-matrix to recover a closely related one satisfying both (\ref{eq:massc}) and the ``right-stochastic'' criterion.

Besides, computations in Section \ref{sec:6} indicate the importance of having steady solutions $\RR^+_* \ni v \mapsto \Psi_{\pm n}(x,v)$ being a $T$-system for any $x>0$ and $n <N \in \NN$. Such an issue appears to be specific to kinetic models like Fokker-Planck equations, involving a local differential operator as the collision mechanism. Indeed, the corresponding stationary boundary-value problem is usually solved by means of a modulation of Sturm-Liouville eigenfunctions, for which the so--called ``Haar property'' (\ref{haar-det}) is not obvious; for instance it may lead to ``generalized polynomials'', like {\it e.g.} exponential ones, for which the equivalent condition (\ref{T-roots}) is not satisfied in general, as recalled in our Appendix \ref{sec:polya}

We now emphasize two possible applications of these schemes, one partly studied, the other being essentially an outlook extending the present work:
\begin{itemize}
\item Some $1+1$ kinetic models of chemotaxis dynamics exhibit (in sharp contrast with diffusive Keller-Segel approximations) {\it bi-stability phenomena of traveling waves}, within certain ranges of parameters. Practical computational simulations were achieved in \cite{CGT} thanks to numerical algorithms based on both $S$-matrices and $\calL$-splines discretizations of diffusive equations \cite{indam}.
\item In \cite[Chap. 6]{cerc} and \cite[Chap. 5]{cerc2006} Cercignani explains how the formalism of ``Case's elementary solutions'' can be extended to linearized BGK models of the Boltzmann equation in $1+1$ dimensions (see also the paper \cite{kcs} and numerical computations in \cite{KRM}). In particular, in \cite[pp. 106--108]{cerc2006}, a (formal) passage from kinetic heat transfer system toward linearized Navier-Stokes-Fourier equations. The discretization of the heat transfer system was presented in {\it e.g.} \cite[Chap. 14]{book}, so that some techniques developed in this paper may shed light onto the corresponding macroscopic behavior.
\end{itemize}

\section*{Acknowledgment} We gladly thank Prof. Christian Krattenthaller (Vienna) for his kind help in the study of the Haar property satisfied by exponential monomials.

This work is supported by French/Italian PICS project {\it MathCell} (CNRS/CNR). NV acknowledges partial support from french ``ANR Blanche'' project Kibord: ANR-13-BS01-0004.

\appendix

\section{Some properties on Case's eigenelements}

In this Appendix, we establish some useful properties on the eigenfunctions defined in Proposition \ref{prop:lambda}.
First we show that the set of Case's eigenfunctions is endowed with the Haar property (see Definition \ref{haar-def}).
\begin{proposition}\label{prop-zeta-gamma}
Denote $\phi_\lambda(v)= \frac{1}{1-\lambda\, v}$ for $\lambda\geq 0$, then the following properties hold:
\begin{itemize}
\item[(i)] Let $0<\lambda_1 <\ldots < \lambda_{K-1}$ and $0<v_1<\ldots<v_K$. We denote $\calV=(v_1,\ldots,v_K)^\top$.
Then, the family $\{\mathbf{1}_{\RR^K}, \phi_{\lambda_1}(\calV), \cdots , \phi_{\lambda_{K-1}}(\calV) \}$ is a basis of $\RR^K$.
\item[(ii)] The set $(\phi_{\lambda})_{\lambda \geq 0}$ is a Markov system on $\RR^+_*$ in the sense of Def. \ref{markov-def}.
\item[(iii)] There exists $\beta\in \RR^K$ and $\gamma\in\calM_{K-1\times K}(\RR)$ such that 
$$
\begin{pmatrix} \beta^\top \\ \gamma \end{pmatrix} = 
\Big(\mathbf{1}_{\RR^K} \quad \phi_{\lambda_1}(\calV) \quad \cdots \quad \phi_{\lambda_{K-1}}(\calV)\Big)^{-1}.
$$
\end{itemize}
\end{proposition}
\begin{proof}These properties are shown by studying convenient polynomials.
\begin{itemize}
\item[(i)] As the family contains $K$ vectors, it suffices to show its linear independence: Assume
$$\exists a_0, a_1, ..., a_K, \qquad 
a_0 \mathbf{1}_{\RR^K} + \sum_{i=1}^{K-1} a_i \phi_{\lambda_i}(\calV) = \mathbf{0}_{\RR^K},
$$
let us show that $a_0=a_1=\ldots=a_{k-1}=0$.
Using the expression of $\phi_\lambda$ and multiplying, 
$$\forall k\in\{1,\ldots,K\}, \qquad 
\left(a_0 + \sum_{i=1}^{K-1} \frac{a_i }{1-\lambda_i v_k}\right)\times \prod_{j=1}^{K-1} (1-\lambda_j v_k) = 0.
$$
Thus, for any $k\in\{1,\ldots,K\}$, the polynomial
$$v \mapsto 
P(v) := a_0 \prod_{j=1}^{K-1} (1-\lambda_j v) + \sum_{i=1}^{K-1} a_i \prod_{j=1,j\neq i}^{K-1} (1-\lambda_j v)
$$
has degree $K-1$, but $K$ roots $\{v_1,\ldots,v_K\}$, so that $P$ is a null polynomial, i.e.
$$
 \forall v\in \RR, \qquad P(v) = 0.
$$  
Identifying the term of higher degree, we deduce that $a_0=0$. Then, taking $v=1/\lambda_i$, $i\in\{1,\ldots,K-1\}$, we obtain $a_i=0$, for all $i\in\{1,\ldots,K-1\}$. 
\item[(ii)] From (i), since the values of both $K \in \NN$ and $\lambda$'s are arbitrary, the set $(\phi_\lambda(v))_{\lambda \geq 0}$ clearly constitutes a Markov system.
\item[(iii)] The point (i) implies that the matrix $\Big(\mathbf{1}_{\RR^K} \quad \phi_{\lambda_1}(\calV) \quad \cdots \quad \phi_{\lambda_{K-1}}(\calV)\Big)$ is invertible.
\end{itemize}
\qed\end{proof}

For our next property, we consider two sets of positive numbers $0<\lambda_1<\ldots<\lambda_{K-1}$
and $0<\mu_1<\ldots<\mu_{K-1}$ with corresponding Case's eigenfunctions $(\phi_\lambda)_\lambda$ 
and $(\phi_\mu)_\mu$. We denote $\gamma_1$, respectively $\gamma_2$,
the corresponding matrices defined in Proposition \ref{prop-zeta-gamma} (iii) for the set $(\lambda_i)_i$,
respectively $(\mu_i)$. We introduce
\begin{align*}
&\zeta_1=\Big(\phi_{\lambda_1}(\calV)-\phi_{\lambda_1}(-\calV), \ldots, \phi_{\lambda_{K-1}}(\calV)-\phi_{\lambda_{K-1}}(-\calV)\Big)  \\
&\zeta_2=\Big(\phi_{\mu_1}(\calV)-\phi_{\mu_1}(-\calV), \ldots, \phi_{\mu_{K-1}}(\calV)-\phi_{\mu_{K-1}}(-\calV)\Big)
\end{align*}
Then, let us denote
$$
\mathcal{H} = 
\begin{pmatrix} \mathbf{I}_K & \zeta_2\gamma_2 - \mathbf{I}_K \\ 
\zeta_1\gamma_1 - \mathbf{I}_K & \mathbf{I}_K \end{pmatrix}.
$$
The following Lemma shed light onto the kernel and the range of $\mathcal{H}$:
\begin{lemma}\label{Hzero}
With the above notations, let us assume moreover that
$$
\forall\, i, \quad \sum_{k=1}^K \omega_k v_k (\phi_{\lambda_i}(v_k)-\phi_{\lambda_i}(-v_k)) = 
\sum_{k=1}^K \omega_k v_k (\phi_{\mu_i}(v_k)-\phi_{\mu_i}(-v_k)) = 0.
$$
Then, the matrix $\mathcal{H}$ is such that:
\begin{itemize}
\item Ker$(\mathcal{H})= 
\mbox{Span}(\mathbf{1}_{\RR^{2K}})$,
\item Im$(\mathcal{H})=
\Big\{Z= (Z_1\ Z_2)^\top,\ Z_i\in \RR^{K} \mbox{ such that } 
\sum_{k=1}^K \omega_k ({Z_1}_k+{Z_2}_{k}) = 0\Big\}$.
\end{itemize}
\end{lemma}
\begin{proof}
\begin{itemize}
\item 
Pick $Y=(Y_1\ Y_2)^\top\in \mbox{Ker}(\mathcal{H})$, then 
$$
Y_1-Y_2 = \zeta_1 \gamma_1 Y_1 = -\zeta_2 \gamma_2 Y_2.
$$ 
Since, from Proposition \ref{prop-zeta-gamma}, the families $\{\mathbf{1}_{\RR^K},\phi_{\lambda_1}(\calV),\ldots,\phi_{\lambda_{K-1}}(\calV)\}$ and \linebreak
$\{\mathbf{1}_{\RR^K},\phi_{\mu_1}(\calV),\ldots,\phi_{\mu_{K-1}}(\calV)\}$ are basis of $\RR^K$, we may write
$$
Y_1 = a_0 + \sum_{\ell=1}^{K-1} a_\ell \phi_{\lambda_\ell}(\calV), \qquad
Y_2 = b_0 + \sum_{\ell=1}^{K-1} b_\ell \phi_{\mu_\ell}(\calV).
$$
By definition of $\zeta_i$ and $\gamma_i$, $i=1,2$, we have 
$$
\zeta_1\gamma_1 Y_1 = \sum_{\ell=1}^{K-1} a_\ell(\phi_{\lambda_\ell}(\calV)-\phi_{\lambda_\ell}(-\calV)), \qquad
\zeta_2\gamma_2 Y_2 = \sum_{\ell=1}^{K-1} b_\ell(\phi_{\mu_\ell}(\calV)-\phi_{\mu_\ell}(-\calV)).
$$
Thus from the equalities $Y_1=Y_2-\zeta_2\gamma_2 Y_2$ and $Y_2=Y_1-\zeta_1\gamma_1 Y_1$, we deduce
\begin{align*}
& a_0-b_0+\sum_{\ell=1}^{K-1} \Big(a_\ell \phi_{\lambda_\ell}(\calV) - b_\ell \phi_{\mu_\ell}(-\calV)\Big) = 0,  \\
& a_0-b_0+\sum_{\ell=1}^{K-1} \Big(a_\ell \phi_{\lambda_\ell}(-\calV) - b_\ell \phi_{\mu_\ell}(\calV)\Big) = 0.
\end{align*}
We now proceed as in the proof of Proposition \ref{prop-zeta-gamma} by introducing the polynomial
\begin{align*}
&v\mapsto Q(v):=  (a_0-b_0)\prod_{i=1}^{K-1}(1-\lambda_i v) \prod_{j=1}^{K-1}(1+\mu_j v)  \\
&+ \sum_{\ell=1}^{K-1} a_\ell \prod_{i=1,i\neq \ell}^{K-1}(1-\lambda_i v) \prod_{j=1}^{K-1}(1+\mu_j v)  
 - \sum_{\ell=1}^{K-1} b_\ell \prod_{i=1}^{K-1}(1-\lambda_i v) \prod_{j=1,j\neq \ell}^{K-1}(1+\mu_j v).
\end{align*}
This is a polynomial of degree $2(K-1)$ which admits the $2K$ roots, $\pm v_1, \ldots, \pm v_K$ (from above equalities).
So it is the null polynomial. Picking the values $v=1/\lambda_\ell$ and $v=-1/\mu_\ell$, $\ell=1,\ldots,K-1$,
we deduce that $a_0=b_0$, $a_\ell=0$, and $b_\ell=0$, for $\ell=1,\ldots,K-1$. Therefore, $Y_1=Y_2=a_0 \mathbf{1}_{\RR^K}$.
\item Consider an element in the range of $\mathcal{H}$, $Z=\mathcal{H} Y$, with $Z=(Z_1\ Z_2)^\top$, $Y=(Y_1\ Y_2)^\top$, $Z_i\in \RR^K$, $Y_i\in \RR^K$, $i=1,2$. Then,
$$
\sum_{k=1}^K \omega_k({Z_1}_k + {Z_2}_k) 
= \sum_{k=1}^K \omega_k v_k \sum_{\ell=1}^{K} \big((\zeta_1 \gamma_1)_{k\ell} {Y_1}_\ell + (\zeta_2 \gamma_2)_{k\ell} {Y_2}_\ell\big).
$$
Applying our assumption, we get 
$$\forall \ell, \qquad 
\sum_{k=1}^K \omega_k v_k (\zeta_1 \gamma_1)_{k\ell}=0, \quad 
\sum_{k=1}^K \omega_k v_k (\zeta_2 \gamma_2)_{k\ell}=0,
$$
so, for any $Z=(Z_1\ Z_2)^\top\in \mbox{Im}(\mathcal{H})$, we have $\sum_{k=1}^K \omega_k ({Z_1}_k + {Z_2}_k) = 0$.
The dimension of $\mbox{Ker}(\mathcal{H})$ is 1, so, thanks to the rank-nullity Theorem, 
equalities are as claimed in Lemma \ref{Hzero}.
\end{itemize}
\qed
\end{proof}

\section{Properties of eigenelements of VFP}

This appendix is devoted to the proof of an 
analogue of Lemma \ref{Hzero} for the VFP case under assumptions
on the set of discrete velocities.
We first define the useful notations.
Let $\psi_\ell^0$, $\ell=0,\ldots,K-1$, be defined as in \eqref{eq:psi0}.
Let us assume that assumptions \eqref{hypV1}, \eqref{hypV2} and \eqref{eq:ortho-psi-dis} on the velocity quadrature hold.
Therefore, there exists $\beta\in\RR^K$ and $\gamma\in\calM_{K-1\times K}(\RR)$ such that
$$
\begin{pmatrix} \beta^\top \\ \gamma \end{pmatrix} = 
\Big(\psi_0^0(\calV) \quad \psi_1^0(\calV) \quad \cdots \quad \psi_{K-1}^0(\calV)\Big)^{-1}.
$$
We introduce $\zeta_\ell := \psi_\ell^0(\calV)-\psi_\ell^0(-\calV)$,
and $\zeta := \big(\zeta_1 \ \ldots \ \zeta_{K-1}\big) \in \mathcal{M}_{K\times K-1}(\RR)$.
Then we denote the matrix
$$
\mathcal{H} = \begin{pmatrix}
\mathbf{I}_K & \zeta\gamma - \mathbf{I}_K  \\
\zeta\gamma - \mathbf{I}_K & \mathbf{I}_K
\end{pmatrix}.
$$

\begin{lemma}\label{Hzero-vfp}
With the above notations, if we assume that \eqref{hypV1}, \eqref{hypV2}, and 
\eqref{eq:ortho-psi-dis} hold. 
Then, we have
\begin{itemize}
\item 
$\mbox{Ker}(\mathcal{H})=\mbox{span}\left(\exp(-\frac{\calV^2}{2\kappa})\right)=\mbox{span }\Big(\psi^0_0(\calV)\Big)$,
\item $\mbox{Im}(\mathcal{H}) = \Big\{Z= (Z_1\ Z_2)^\top,\ Z_i\in \RR^{K} \mbox{ such that } 
\sum_{k=1}^K \omega_k ({Z_1}_k+{Z_2}_{k}) = 0\Big\}$.
\end{itemize}
\end{lemma}
\begin{proof}
We proceed as in the proof of Lemma \ref{Hzero}.
\begin{itemize}
\item Let $Y=(Y_1\ Y_2)^\top\in \mbox{Ker}(\mathcal{H})$, then 
$$
Y_1-Y_2 = \zeta \gamma Y_1 = -\zeta \gamma Y_2.
$$ 
By assumption \eqref{hypV1}, the family $\{\psi_0^0(\calV),\, \psi_1^0(\calV), \ldots,\, \psi_{K-1}^0(\calV)\}$ 
is a basis of $\RR^K$, then, we may write
$$
Y_1 = \sum_{\ell=0}^{K-1} a_\ell \psi^0_{\ell}(\calV), \qquad
Y_2 = \sum_{\ell=0}^{K-1} b_\ell \psi^0_{\ell}(\calV).
$$
Simple computations using the definition of $\zeta$ and $\gamma$ and recalling that $\psi_0^0(\calV)=\exp(-\frac{\calV^2}{2\kappa})$, give
$$
\zeta\gamma Y_1 = \sum_{\ell=1}^{K-1} a_\ell(\psi_\ell^0(\calV)-\psi_{\ell}^0(-\calV)), \qquad
\zeta\gamma Y_2 = \sum_{\ell=1}^{K-1} b_\ell(\psi_\ell^0(\calV)-\psi_{\ell}^0(-\calV)).
$$
Thus the equalities $Y_1=Y_2-\zeta\gamma Y_2$ and $Y_2=Y_1-\zeta\gamma Y_1$ imply
\begin{align*}
& (a_0-b_0)\exp(-\frac{\calV^2}{2\kappa})+\sum_{\ell=1}^{K-1} \Big(a_\ell \psi_\ell^0(\calV) - b_\ell \psi_\ell^0(-\calV)\Big) = 0,  \\
& (a_0-b_0)\exp(-\frac{\calV^2}{2\kappa})+\sum_{\ell=1}^{K-1} \Big(a_\ell \psi_\ell^0(-\calV) - b_\ell \psi_\ell^0(\calV)\Big) = 0.
\end{align*}
From assumption \eqref{hypV2}, we deduce that $a_0=b_0$, $a_\ell=0$, $b_\ell=0$, for $\ell=1,\ldots,K-1$. As a consequence $Y_1=Y_2=a_0 \psi_0^0(\calV)$.

\item Consider an element in the range of $\mathcal{H}$, $Z=\mathcal{H} Y$, with $Z=(Z_1\ Z_2)^\top$, $Y=(Y_1\ Y_2)^\top$, $Z_i\in \RR^K$, $Y_i\in \RR^K$, $i=1,2$. Then,
$$
\sum_{k=1}^K \omega_k({Z_1}_k + {Z_2}_k) 
= \sum_{k=1}^K \omega_k v_k \sum_{\ell=1}^{K} (\zeta \gamma)_{k\ell} \big( {Y_1}_\ell + {Y_2}_\ell\big).
$$
Applying our assumption, we get 
$$\forall \ell, \qquad 
\sum_{k=1}^K \omega_k v_k (\zeta \gamma)_{k\ell}=0,
$$
so, for any $Z=(Z_1\ Z_2)^\top\in \mbox{Im}(\mathcal{H})$, we have $\sum_{k=1}^K \omega_k ({Z_1}_k + {Z_2}_k) = 0$.
The dimension of $\mbox{Ker}(\mathcal{H})$ is 1, so, thanks to the rank-nullity Theorem, rank$(\mathcal{H})=K-1$, which allows to conclude the proof.
\end{itemize}
\qed
\end{proof}

\section{Some properties of exponential polynomials}

\subsection{Elementary proof of the P\'olya-Szeg\"o estimate} \label{sec:polya}

Hereafter, following \cite[page 10]{karlin}, we establish {\it by induction} a simple bound on the number of real roots of an exponential polynomial; for various extensions, see \cite{marc,wol}
$$
\forall n \in \NN, \qquad
f_n(x)= \sum_{i=0}^{n-1} P_i(x) \exp(\mu_i\, x), \qquad \mu_i \in \RR,\quad \mbox{deg}(P_i)=k_i.
$$
We intend to show that, for any $n \in \NN$, $f_n$ admits {\it at most} $N_n-1$ roots, where
\begin{equation}\label{PS-bound}
N_n=\left(\sum_{i=0}^{n-1} (1+k_i)\right).
\end{equation}
We use an induction on $n$:
\begin{itemize}
\item for $n=1$, the exponential polynomial reads $f_1(x)=P_0(x) \exp(\mu_0\, x)$ so it admits at most $k_0=N_1-1$ roots.
\item Assume the property (\ref{PS-bound}) holds for $f_n$, so that it admits at most $N_n-1$ real roots. Let $M$ be the number ot real roots of $f_{n+1}$, and define
\begin{align*}
\forall x \in \RR, \qquad 
f_{n+1}(x)\exp(-\mu_{n}\,x)&=\sum_{i=0}^n P_i(x)\exp((\mu_i - \mu_{n})\,x)\\
&=P_n(x)+\sum_{i=0}^{n-1} P_i(x)\exp((\mu_i - \mu_{n})\,x).
\end{align*}
By the classical Rolle's theorem for smooth functions, its $(1+k_n)^{th}$ derivative
\begin{align*}
\forall x \in \RR, \qquad 
g_{n+1}(x)&= \frac{d^{(1+k_n)}}{dx^{(1+k_n)}}[f_{n+1}(x)\exp(-\mu_{n}\,x)] \\
&=\sum_{i=0}^{n-1}  \frac{d^{(1+k_n)}}{dx^{(1+k_n)}}[P_i(x)\exp((\mu_i - \mu_{n})\,x)],
\end{align*}
admits at least $M-(1+k_n)$ roots. But since $g_{n+1}$ is an exponential polynomial to which (\ref{PS-bound}) applies, it comes that
$$
M-(1+k_n) \leq N_n -1, \qquad \mbox{ so that }\quad M \leq N_n + (1+k_n) -1 := N_{n+1} -1.
$$ 
\end{itemize}

\subsection{Haar property for exponential monomials}

Although the former estimate suggests that exponential polynomials is not a Chebyshev $T$-system, {\it non-negative exponential monomials} do satisfy the Haar property on $(0,+\infty)$:
\begin{theorem}[Krattenthaller, \cite{Krat}]
\label{th:alternant}
Let $(x_0,x_1,\dots,x_{n-1}) \in \RR_+^n$, $(y_0,y_1,\dots,y_{n-1}) \in \RR^n_+$ 
be \underline{non-negative} with $y_0<y_1<\dots<y_{n-1}$. 
Moreover, let $(z_0,z_1,\dots,z_{n-1}) \in \NN^n$ be  
\underline{non-negative integers} with $z_0<z_1<\dots<z_{n-1}$. 
The generalized Vandermonde determinant,
\begin{equation} \label{eq:1} 
\det_{0\le i,j<n}\left(e^{y_jx_i}x_i^{z_j}\right)
\end{equation}
vanishes if and only if two of the $x_i$'s are equal to each other.
\end{theorem}
The proof of this result relies on an expansion of the exponential
and the use of Schur functions \cite{KratBW,LindAA,MacdAC}.


\end{document}